\documentclass[11pt, reqno]{amsart}
\usepackage{amsfonts,amssymb,latexsym,amsmath, amsxtra,url}
\usepackage{verbatim}
\usepackage{enumitem}
\usepackage{capt-of}
\usepackage{tikz}
\usepackage{cancel}
\usepackage{bm}

\allowdisplaybreaks
\theoremstyle{plain}
\newtheorem{theorem}{Theorem}[section]
\newtheorem*{theorem*}{Theorem}
\newtheorem{corollary}[theorem]{Corollary}
\newtheorem{conjecture}[theorem]{Conjecture}

\newtheorem{lemma}[theorem]{Lemma}
\newtheorem{proposition}[theorem]{Proposition}
\newtheorem*{conjecture*}{Conjecture}
\newtheorem*{proposition*}{Proposition}

\theoremstyle{definition}
\newtheorem*{remark}{Remark}
\newtheorem*{remark*}{Remark}
\newtheorem*{remarks*}{Remarks}

\newtheorem*{definition*}{Definition}

\newtheorem*{example*}{Example}
\newtheorem*{examples*}{Examples}
\theoremstyle{remark}

\numberwithin{equation}{section}

\newcommand{\R}{\mathbb R}
\newcommand{\N}{\mathbb N}
\newcommand{\Z}{\mathbb Z}
\newcommand{\C}{\mathbb C}
\newcommand{\HH}{\mathbb H}

\newcommand{\Q}{{\mathbb Q}}

\def\({\left(}
\def\){\right)}

\def\calG{\mathcal{G}}
\def\calF{\mathcal{F}}
\def\calL{\mathcal{L}}
\def\calR{\mathcal{R}}

\def\lp{\left(}
\def\rp{\right)}

\def\a{\alpha}

\def\z{\zeta}

\def\t{\tau}
\def\SL{\mathrm{SL}}

\allowdisplaybreaks

\def\del{  \partial}

\renewcommand{\pmod}[1]{\,\,({\rm mod}\,\,{#1})}

\newcommand{\wh}[1]{\widehat{#1}}

\newcommand{\sgn}{\text{sgn}}

\def\lp{\left(}
\def\rp{\right)}

\begin{document}

\title[Graph Schemes, Graph Series, and Modularity]{Graph Schemes, Graph Series, and Modularity}

\author{Kathrin Bringmann}
\address{University of Cologne, Department of Mathematics and Computer Science, Weyertal 86-90, 50931 Cologne, Germany}
\email{kbringma@math.uni-koeln.de}

\author{Chris Jennings-Shaffer}

\address{Department of Mathematics, University of Denver, 2390 S. York St. Denver, CO 80208}
\email{christopher.jennings-shaffer@du.edu}

\author{Antun Milas}
\address{Department of Mathematics, SUNY-Albany, Albany NY 12222}
\email{amilas@albany.edu}

\begin{abstract} 
To a simple graph we associate a so-called graph series, which can be viewed as the Hilbert--Poincar\'e series of a certain infinite jet scheme. We study new $q$-representations and examine modular properties of several examples including Dynkin diagrams of finite and affine type. 
Notably, we obtain new formulas for graph series of type $A_7$ and $A_8$ in terms of ``sum of tails" series, and of type $D_4$ and $D_5$ in the form of indefinite theta functions of signature $(1,1)$. We also study examples related to sums of powers of divisors corresponding to $5$-cycles. For several examples of graphs, we prove that graph series are so-called mixed quantum modular forms.
\end{abstract}
\maketitle

\section{Introduction and statement of results}

The interplay between (multi) $q$-hypergeometric series and modular forms has been studied extensively over the years.
In theoretical physics and representation theory, $q$-hypergeometric series famously appeared in connection with ``fermionic'' formulas of characters of rational 
conformal field theories. If their summands involve a positive definite quadratic form they are sometimes called  Nahm sums \cite{NRT}.
Important examples of this type emerged much earlier as ``sum sides'' of several classical $q$-series identities (e.g. Andrews--Gordon identities).
More recently, Nahm sums have appeared in several areas including vertex algebras \cite{Kawasetsu,Li,MP, Penn}, DAHA \cite{CF}, wall-crossing phenomena and quivers \cite{CNV,KS,LM1}, in connection to colored HOMFLY-PT polynomials \cite{Kucharski}, and in formulas for the ``tail" of colored Jones polynomials \cite{BO,GL,KO}.

To define a Nahm sum we choose a positive definite quadratic form $Q : \mathbb{Z}^r \rightarrow \mathbb{Q}$ and consider 
\begin{equation} \label{Nahm}
\sum_{\bm{n}\in\N_0^r} \frac{q^{Q({\bm n})+\bm{b} \cdot {\bm n}+c}}{(q)_{n_1} \cdots (q)_{n_r}},
\end{equation}
where ${\bm n}=(n_1,...,n_r)$, $(a;q)_n=(a)_n:=\prod_{j=0}^{n-1}(1-aq^j)$, $\bm{b} \in \mathbb{Q}^r$, and $c \in \mathbb{Q}$.
One important aspect in the study of these series is to investigate their modularity by considering
a suitable constant term in the asymptotic expansion (see \cite{CGZ}).
Another interesting question is that of how to express (\ref{Nahm}) as a ``bosonic'' sum, which may or may not be modular. Prominent examples of Nahm sums come from Cartan matrices of classical type and their inverses, as they are all expected to be modular for specific choices of $\bm b$ and $c$ (see \cite{WZ} for some recent results).
       
In this work, we study closely related but different $q$-hypergeometric series coming from jet schemes and vertex algebras. Our aim is to investigate {\em graph series} defined as
\begin{equation}\label{hg}
H_{\Gamma}(q):=\sum_{\bm{n}\in\N_0^r} \frac{q^{\frac12 {\bm n} C_\Gamma {\bm n}^T + n_1+\cdots + n_r }}{(q)_{n_1} \cdots (q)_{n_r}},
\end{equation}
where $C_\Gamma$ is the (symmetric) adjacency  matrix of a graph $\Gamma$. For many graphs, including all simple graphs, this series is not a Nahm sum (the adjacency matrix is not positive definite!). Another issue is that the matrix $C_\Gamma$ can be singular, and indeed often is, so it is important that we include the linear term in the exponent.
  The main motivation for studying graph series comes from two sources. As explained in Section 2, for a given graph $\Gamma$, there is a  graded commutative algebra $J_\infty(R)$, the ring of functions of the infinite jet scheme $X_\infty$, where $X={\rm Spec}(R)$, whose Hilbert series is given by (\ref{hg}). This infinitely-generated algebra is closely related to a certain principal subspace vertex algebra constructed from the adjacency matrix of $\Gamma$ \cite{Kawasetsu,Li, MP,Penn}, in the sense that its character is the Hilbert series of $J_\infty(R)$.

In this paper we make the first steps in addressing modularity properties of graph series. 
We focus on examples coming from Dynkin diagrams of finite type, as well as several affine Dynkin diagrams. We show that in many examples their modular properties can be quite interesting in spite of the simplicity of the graph. Since the form of (\ref{hg}) does not give any clues about modularity, we first must obtain suitable bosonic or combinatorial formulas in terms of functions whose modular properties are transparent. For graph series of type $A$ (paths), studied in \cite{JM}, our first result is two new representations for graph series of type $A_7$ and $A_8$ (i.e., paths with seven and eight nodes,  respectively). 
\begin{theorem}\label{A7.A8}
We have 
\begin{align*}
H_{A_7}(q)&=\frac{q^{-1}}{(1-q)(q)_\infty^4} \left(-1+(q)_\infty D(q) +G(q)+(q)_\infty \right), \\
H_{A_8}(q)&= \frac{q^{-2}}{(q)_\infty^4} \biggl(-1+ (q)_\infty  + 3  D(q)  - 2 G(q)\biggr),
\end{align*}
where $D(q):=\sum_{n \geq 1} \frac{q^n}{1-q^n}$ and $G(q):=\sum_{n \geq 0} ((q)_n -(q)_\infty)$.
\end{theorem}
Following Zagier \cite{ZaQ}, \textit{quantum modular forms} are functions $f: \mathcal{Q} \to \C \ (\mathcal{Q} \subset \Q)$ whose \textit{obstruction to modularity}
$$
f(x)-(cx+d)^{-k} f \lp \frac{ax+b}{cx+d} \rp, \quad  \left(\begin{matrix} a&b\\c&d\end{matrix}\right) \in \SL_2 \lp \Z \rp
$$
is ``nice" (see Subsection \ref{modularity} for more details).
Combined with known $q$-series representations for $H_{A_j}(q)$, $1 \leq j \leq 6$, given in \cite[Section 7]{JM}, we obtain the following result. 
\begin{corollary} For every $j$, $1 \leq j \leq 8$, $H_{A_j}(q) \in \mathbb{C}[\frac{1}{1-q},q^{-1},\frac{1}{(q)_\infty},D(q),G(q)]$. More precisely, the series $q (q)_\infty^2 H_{A_4}(q)$ is a holomorphic quantum modular form of weight one, while $q (q)^{3}_\infty H_{A_5}(q)$ and $q (1-q) (q)^{3}_\infty H_{A_7}(q)$ are quantum modular forms of weight $\frac32$.
  \end{corollary}
 This indicates the possibility of accommodating all $A$-type graph series inside a finitely generated ring. We also investigate asymptotic properties of these graph series as $t \to 0^+$, where $q=e^{-t}$ (see Proposition \ref{seriesA}). Such analysis is important from the geometric viewpoint as we would like to understand the growth of coefficients of the Hilbert series of $J_\infty(R)$.

Graph with cycles are more complicated to analyze. However, for several examples related to $5$-cycles we obtain elegant formulas.
\begin{theorem} \label{pentagons} Let $C_5$ be a $5$-cycle graph and $\Gamma_8$ the graph in Figure 1. Then
\begin{align}\label{HC5}
H_{C_5}(q)&=\frac{q^{-1}}{(q)_\infty} \sum_{n \geq 1} \frac{nq^n}{1-q^n}, \\ 
\label{g8}
H_{\Gamma_8}(q)&=\frac{q^{-1}}{(q)^2_\infty} \sum_{n \geq 1} \frac{n^2 q^n}{1-q^n}.
\end{align}
\end{theorem}
 
 For several examples of graphs of $D$ and $E$-type, due to a trivalent node, we get graph series with somewhat different combinatorial and modular properties. For this recall that \textit{mixed mock modular forms} are linear combinations of modular forms multiplied by mock modular forms (holomorphic parts of certain non-holomorphic automorphic objects). Here is our main result in this direction.
\begin{theorem}\label{thm:1.3} We have
\begin{align}\label{HD4}
H_{D_4}(q)&=\frac{1}{(q)_\infty^4} \sum_{n,m \geq 0} (-1)^{n+m} (2n+1) q^{\frac{n^2}{2}+\frac{3m^2}{2}+2nm+ \frac{3n}{2}+\frac{5m}{2}}, \\ \label{HD5.2}
H_{D_5}(q)&=\frac{1}{(q)^5_\infty} \lp \sum_{n,m \geq 0} - \sum_{n,m < 0} \rp (-1)^{n+1} (n+1)^2 q^{\frac{n^2+3n}{2} +3nm+3m^2+4m}, \\ \label{HE6}
H_{E_6}(q)&=\frac{q^{-1}}{(q)^2_\infty} \sum_{n \geq 1} \frac{nq^n}{1-q^n}.
\end{align}
All three series are mixed mock modular forms. 
\end{theorem}

Although $q$-series associated to graphs with multiple edges do not directly relate to Hilbert series of jet algebras, they can be viewed as characters of certain principal subspaces (see Section 2). In this paper we obtain several $q$-series identities for graph series of type $B_2$ and $B_3$ (ignoring the orientation) and related ``coset" series (see Proposition \ref{prop:6.1}). For several cases we obtain mixed quantum modular forms. This indicates a possible connection with quantum invariants of $3$-manifolds and knots where similar series appear \cite{Hikami,LZ,Za}. We also investigate several examples of graph series associated to affine Dynkin diagrams whose modular properties seem 
more intricate (see Section 8). 

The paper is organized as follows. In Section 2 we present the main concepts and definitions, along with preliminary results. In particular, we define the notion of the jet scheme  of an affine scheme $X$, briefly discuss principal subspaces associated to lattices and graphs, and present a few results on mock modular forms. In Section 3, we give two new results on the graph series of type $A_7$ and $A_8$ (see Theorem \ref{A7.A8}) as well as study modular properties and the asymptotic behavior of graph series of type $A_n$, $1 \leq n \leq 8$ (see Proposition \ref{seriesA}). In Section 4 we analyze certain graph series coming from $5$-cycles. Section 5 is concerned with the  $D_4$ graph series. The main result here is a new bosonic $q$-representation of the graph series in terms of an indefinite theta function of signature $(1,1)$ (see \eqref{HD4}). Moreover, we show that this graph series is a mixed mock modular form. In Section 6, we perform a similar analysis for the $D_5$ graph,  but this time we use (mixed) mock modular forms to obtain an indefinite theta function representation (see Theorem \ref{theorem:5.1}). 
Section 7 deals with a few examples of graphs with multiple edges of Dynkin type $B_2$ and $B_3$. Finally, in Section 8, we investigate graph series from the $E_6$ Dynkin diagram, $3$-cycles (affine type $A_{2}^{(1)}$), and affine Dynkin graphs of type $D_5^{(1)}$ and $E_6^{(1)}$ .

\section*{Acknowledgements}
The first author has received funding from the European Research Council (ERC) under the European Union’s Horizon 2020 research and innovation programme (grant agreement No. 101001179). 


\section{Preliminaries}

In this part we introduce the main objects of study and present some preliminary results. 

\subsection{Graph schemes and graph series} 

This subsection outlines a construction of the arc space.  
We define the arc space  and the arc algebra (or algebra of infinite jets) of a finitely generated commutative ring $R$.  As usual, let $\mathbb{C}[x_1,x_2, \ldots, x_{\ell}]$ denote the polynomial algebra in $x_j$ $(1 \leq j \leq \ell)$, let $f_{1},f_{2}, \ldots, f_{n}$ be a set of polynomials, and define the quotient algebra \begin{align*}R:=\frac{ \mathbb{C}[x_1,x_2, \ldots, x_{\ell}]}{(  f_{1},f_{2}, \ldots, f_{n})}.
\end{align*}
We now introduce new variables $x_{j,(-1-k)}$ for $k \in \{0,\ldots,m\}$.  We define a derivation $T$ on \[\mathbb{C}[x_{j,(-1-k)}:0\leq k\leq m,\; 1\leq j\leq \ell], \] by letting \begin{equation*}
  T\left(x_{j,(-1-k)}\right) :=
    \begin{cases}
      (-1-k)x_{j,(-k-2)} & \text{for $k\leq m-1$},\\
      0 & \text{for $k=m$}.\\
    \end{cases}
 \end{equation*}
We also identify $x_{j}$ with $x_{j,(-1)}$. Set 
\begin{align*}
R_{m}:=\frac{\mathbb{C}[x_{j,(-1-k)}:0\leq k\leq m,\; 1\leq j\leq \ell]}{( T^{j}f_{k}:1 \leq k \leq n, j\in\mathbb{N}) } ,  
\end{align*}
 the \textit{algebra of $m$-jets} of $R$.
The {\it arc algebra} of $R$ is defined as the direct limit 
\begin{align*} 
J_{\infty}(R)&:=\displaystyle\lim_{\underset{m}{\rightarrow}}R_{m}=\frac{\mathbb{C}[x_{j,(-1-k)}:0\leq k,\; 1\leq j\leq \ell]}{( T^{j}f_{k}:k=1,\ldots n,\; j\in\mathbb{N}) }.
\end{align*}
The scheme $X_\infty=\displaystyle\lim_{\underset{m}{\leftarrow}} X_m$ (another notation is $J_\infty(X)$), where $X_m={\rm Spec}(R_m)$, is called the \textit{ infinite jet scheme}  of $X={\rm Spec}(R)$.  By construction, the arc algebra $J_\infty(R)$ is a differential  commutative algebra.
If $R$ is graded, then $J_\infty(R)$ is also graded, and we can define its {\it Hilbert--(Poincare) series} as: \[HS_{q}(J_{\infty}(R)):=\sum_{m \in  \mathbb{Z}} {\rm dim}\left(J_{\infty}\left(R\right)_{(m)}\right) q^m.\]
We introduce a grading on $R$ by letting ${\rm deg}(x_{j,(-1-k)}):=k+1$.

Let $V$ be a vertex algebra. Provided that $V$ is  strongly finitely generated, we obtain a surjective map from $J_\infty(R_V)$ to ${\rm gr}(V)$, the associated graded algebra of $V$,  where $R_V$ is Zhu's commutative algebra of a vertex algebra $V$. A vertex algebra for which this map is injective is said to be {\em classically free} \cite{ ARE,Arakawa, Li, LM1,LM2}. Although this notion is relatively new, non-trivial examples of vertex algebras with this property appeared earlier in the framework of principal subspaces \cite{CLM,CLM2,FS} and Virasoro minimal models \cite{FF}.
Recall a result announced in \cite[Section 9]{JM} motivated by \cite{MP}, with full details provided in \cite{Li-Thesis, Li}. 
\begin{theorem} Let $\Gamma$ be any simple graph with $r$ nodes and without multiple edges. 
Consider the scheme $X_\Gamma$ defined by the quadratic equation $x_kx_j=0$, if  $k$ and $j$ are adjacent i.e., $(k,j) \in E(\Gamma)$, and 
denote by 
$J_\infty(X_\Gamma)$ the infinite jet algebra of  $X_\Gamma$.
Then the Hilbert series of $J_\infty(X)$, with $\deg(x_j)=1$, is given by 
\begin{equation*}
H_{\Gamma}(q)= \sum_{\bm n \in \N_0^r} \frac{q^{\frac{1}{2} \sum_{(k,j) \in E(\Gamma)} n_k n_j+n_1+\cdots + n_r}} {(q)_{n_1} \cdots (q)_{n_r}}.
\end{equation*}
\end{theorem}

Another important result in this context is the following theorem (see \cite{Li} and also \cite{MP,Penn}).
\begin{theorem} \label{principal}
For any $\Gamma$ (not necessarily simple), the $q$-series $H_{\Gamma}(q)$ computes the character of the principal subspace $W_\Gamma$ of a lattice vertex algebra associated to the adjacency 
matrix of $\Gamma$, equipped with a certain grading.
\end{theorem}
We also note that the previous theorem can be used, with slight modifications, for graphs with loops. In that case we need to adjust the linear term of the $q$-exponent of $H_{\Gamma}(q)$ (see \cite{LM1} for some examples). We now give some modular examples.

\begin{examples*}
\phantom{abc}

\noindent (1) (Affine space $\mathbb{A}^\ell$) For $R=\mathbb{C}[x_1,...,x_\ell]$, we have$$
J_\infty(R)=\mathbb{C} \left[x_{1,(-1)},...,x_{1,(-n)},...,x_{\ell,(-1)},...,x_{\ell,(-n)},...\right].
$$
Since ${\rm deg}(x_{r,(-j)})=j$, we immediately get 
$$HS_q(J_\infty(R))=\frac{1}{(q)_\infty^\ell}.$$

\noindent (2) (Union of two lines)
 Let $R=\mathbb{C}[x,y]/(xy)$.
Then by \cite{Arakawa, M, JM} we have 
$$HS_q(J_\infty(R))=\sum_{n_1,n_2 \geq 0} \frac{q^{n_1 n_2 +n_1+n_2}}{(q)_{n_1} (q)_{n_2}}=\frac{1}{(1-q)(q)_\infty}.$$

\noindent (3) (``Fat" point) Let $R=\mathbb{C}[x]/(x^2)$.
Then by \cite{Gorsky, M, CLM, FS} we have
$$HS_q(J_\infty(R))=\sum_{n \geq 0} \frac{q^{n^2}}{(q)_n}=\frac{1}{(q;q^5)_\infty(q^4;q^5)_\infty},$$
where the last equality is the first Rogers--Ramanujan identity.
\end{examples*}

\subsection{Modularity results}\label{modularity}
In this subsection we record some modularity results required for this paper. We start with the transformation laws of the \textit{Dedekind $\eta$-function} $\eta(\tau):=q^\frac{1}{24} (q)_\infty :$
$$
 \eta(\tau+1)=e^\frac{\pi i}{12} \eta(\tau), \quad \eta\left(-\frac 1\tau\right)= \sqrt{-i \tau} \eta(\tau).
$$
We define the {\it weight two Eisenstein series} $(q:=e^{2\pi i \tau})$
\begin{equation}\label{E2.1}
E_2(\t) := 1-24 \sum_{n\ge1} \sum_{d|n} d q^n.
\end{equation}
This function is not quite a modular form, but transforms with an additional term under inversion. To be more precise, we have
\begin{equation}\label{E2}
	E_2\left(-\frac 1\t\right) = \t^2 E_2(\t) + \frac{6\t}{\pi i}.
\end{equation}
We also require the following representation of $E_2$ as a Lerch-type sum.
Although the next result is probably known, we include a proof for the sake of completeness.
\begin{lemma}\label{AL_E2}
	We have 
	$$
	\frac{1-E_2(\t)}{24} = \sum_{n \geq 1} \frac{q^n}{(1-q^n)^2}=\sum_{n \geq 1}
	\frac{(-1)^{n+1} (1+q^n) q^{\frac{n(n+1)}{2}} }{(1-q^n)^2}.
	$$
\end{lemma}
\begin{proof} The first equality is just rewriting \eqref{E2.1}. Using l'Hospital's rule twice the second equality follows from taking the limit $\zeta \to 1$ of the following identity (valid for $0 <|\zeta|<|q|<1$):
	\begin{align} \label{t-identity_1}
		& -\frac{\zeta}{(1-\zeta)^2} +\frac{\zeta (q)_\infty^2}{\lp 1-\zeta \rp (\zeta)_\infty (\zeta^{-1}q)_\infty } =\sum_{n \geq 1} \frac{(-1)^{n+1} (1+q^n) q^{\frac{n(n+1)}{2}}}{(1-\zeta q^{n})(1-\zeta^{-1}q^n)}.
	\end{align}
	Note that \eqref{t-identity_1} is implied by the well-known formula (see for instance \cite{Lerch})
	\begin{equation*}\label{lerch}
		\frac{(q)_\infty^2}{(\zeta )_\infty (\zeta^{-1}q)_\infty}=\sum_{n \in \mathbb{Z}} \frac{(-1)^{n} q^{\frac{n(n+1)}{2}}}{1-\zeta q^n}. \qedhere
	\end{equation*}
\end{proof}

We also require modularity properties of
\begin{align*}
	\mathcal{F}(\tau) := \sum_{n\in\Z \setminus \{0\}} \frac{(-1)^{n+1}q^{\frac{3n(n+1)}{2}}}{\lp 1-q^n\rp^2}.
\end{align*}
For this define ($\tau=u+iv$) the {\it completion of $\calF$} as
\begin{align*}
	\widehat{\mathcal{F}}(\tau) := \mathcal{F}(\tau) - \frac{1}{24} + \frac{E_2(\tau)}{8}
	-\frac{3\eta(\t)}{2\sqrt{\pi}} \sum_{n\in\Z-\frac 16} (-1)^{n+\frac 16} \lvert n\rvert \Gamma\left(-\frac 12, 6\pi n^2v\right) q^{-\frac{3n^2}{2}},
\end{align*}
where the {\it incomplete gamma function} is defined as $\Gamma(\alpha,x):=\int_{x}^{\infty} e^{-t} t^{\alpha-1} dt$
for $x \in \R^+$ and $\a\in\R$.
We then have the following modularity result.
\begin{proposition}\label{Duke}
	We have for $\left(\begin{smallmatrix} a&b\\c&d\end{smallmatrix}\right) \in \SL_2 \lp \Z \rp$,
	\begin{align*}
			\widehat{\mathcal{F}} \lp \frac{a\tau + b}{c \tau +d} \rp = (c\t+d)^2 \widehat{\mathcal{F}}(\t).
	\end{align*}
\end{proposition}
\begin{proof}
	Set
	\begin{align*}
		r(\tau):= r^+ (\t) +r^- (\tau)
	\end{align*}
with
	\begin{align*}
	r^+(\tau)&:= \frac{2\pi i }{\eta(\tau)} \lp \mathcal{F}(\t) - \frac{1}{24}+\frac{E_2(\t)}{8} \rp, \\ 
	r^-(\tau)&:= \frac{1}{2\pi i} \left[\frac{\partial}{\partial z} \left(\zeta^{-1} q^{-\frac 16} R(3z+\tau;3\tau) e^{-\frac{\pi^2}{2} E_2(\tau) z^2}\right)\right]_{z=0},
\end{align*}
where ($z=x+iy$, $\zeta:=e^{2\pi iz}$)
\begin{align*}
	R(z;\tau):=\sum_{n\in \Z + \frac 12} (-1)^{n-\frac 12} \lp \sgn(n) -E\lp \lp n+\frac{y}{v}\rp \sqrt{2 v} \rp \rp q^{-\frac{n^2}{2}}e^{-2\pi i n z},
\end{align*}
with $E(x):= 2 \int_{0}^{x}e^{-\pi t^2}dt$.
By using \cite{35} we obtain that for $\left(\begin{smallmatrix} a&b\\c&d\end{smallmatrix}\right) \in \SL_2 \lp \Z \rp$,
\begin{align}\label{etatran}
	\eta \lp \frac{a\t+b}{c\t+d} \rp r \lp \frac{a\t+b}{c\t+d} \rp 
	=(c\tau+d)^2 \eta(\t) r(\t).
\end{align}
We compute
\begin{align*}
	r^-(\tau)   
	&= \frac{1}{2\pi i} q^{-\frac 16} \sum_{n\in\Z+\frac 12} (-1)^{n-\frac 12}  \left[\frac{\partial}{\partial z} \left(\left(\sgn(n) - E\left(\left(n+ \frac 13 +\frac yv  \right) \sqrt{6v}\right)\right) \zeta^{-3n-1}\right)\right]_{z=0} q^{-\frac{3n^2}{2}-n}\\
	& =-\frac{1}{2\pi i} \sum_{n\in\Z-\frac 16} (-1)^{n+\frac 16}  \left[\frac{\partial}{\partial z} \left(\left(\sgn(n) - E\left(\left(n+ \frac yv\right) \sqrt{6v} \right)\right) \zeta^{-3n}\right)\right]_{z=0} q^{-\frac{3n^2}{2}}.
\end{align*}
	We now use the identities
	\begin{align}%
		E'(x)&=2e^{-\pi x^2}, \quad  
		E(x)=\sgn(x) \lp 1- \frac{1}{\sqrt{\pi}} \Gamma\lp \frac 12 , \pi x^2 \rp \rp,\quad
		\label{identity1.1} 	\\
		\Gamma\lp \frac 12 , x \rp &= - \frac 12 \Gamma \lp - \frac 12 , x \rp + x^{- \frac 12} e^{-x}\label{identity1.2} 
	\end{align}
	to obtain that 
	\begin{equation*}
	\left[\frac{\partial}{\partial z} \left(\left(\sgn(n) - E\left(\left(n+ \frac yv\right) \sqrt{6v} \right)\right) \zeta^{-3n}\right)\right]_{z=0} = 3\sqrt{\pi} i \lvert n\rvert \Gamma\left(-\frac 12, 6\pi n^2v\right).
	\end{equation*}
	This gives that
	\begin{align*}
		r^-(\tau) = - \frac{3}{2\sqrt{\pi}} \sum_{n\in\Z - \frac 16} (-1)^{n+\frac 16} |n| \Gamma\lp -\frac 12 , 6\pi n^2v \rp q^{-\frac{3n^2}{2}}.
	\end{align*}
Thus we have that $\widehat{\mathcal{F}}(\tau) =\eta(\tau) r(\tau)$.
The claim then follows from \eqref{etatran}.
\end{proof}

We also require certain indefinite theta functions, considered by Zwegers in his thesis \cite{ZW2}. We let $A$ be a symmetric $r\times r$ matrix with integral coefficients that is non-degenerate, let $Q(\bm{x}) := \frac 12 \bm{x}^TA\bm{x}$ be the corresponding quadratic form, and let $B(\bm{x},\bm{y}) := \bm{x}^TA\bm{y}$ be the associated bilinear form. We assume that $Q$ has signature $(r-1,1)$. Fix $\bm{c_0}\in\R^r$ and let
\[
C_Q := \left\{\bm{c}\in\R^r : Q(\bm{c}) < 0, B(\bm{c},\bm{c_0}) < 0\right\}.
\]
For $\bm{c_1},\bm{c_2}\in C_Q$, set
\[
\varrho(\bm{n}) = \varrho_A^{\bm{c_1},\bm{c_2}}(\bm{n};\t) := E\left(\frac{B(\bm{c_1},\bm{n})}{\sqrt{-Q(\bm{c_1})}} \sqrt{v}\right) - E\left(\frac{B(\bm{c_2},\bm{n})}{\sqrt{-Q(\bm{c_2})}} \sqrt{v}\right).
\]
Then define
\begin{equation*}
	\Theta_A \lp \bm z ;\tau \rp 
	=\Theta_A^{\bm{ c_1}, \bm{ c_2}} \lp \bm z ;\tau \rp 
	:= \sum_{\bm n \in \Z^r}
	\varrho \lp \bm n + \frac{\text{Im}(\bm z)}{v} \rp e^{2\pi i B(\bm n, \bm z)} q^{Q(\bm n)}.
\end{equation*}
We have the following properties (see Proposition 2.7 of  \cite{ZW2}).
\begin{proposition}\label{prop:elliptic}
\end{proposition} 
\vspace{-.25cm}
\it
\begin{enumerate}[leftmargin=*,label=\textnormal{(\arabic*)}]
\item We have 
\begin{align*}\label{invT}
	\Theta_A \lp \frac{\bm z}{\tau}; -\frac{1}{\tau} \rp= \frac{i}{\sqrt{-\det(A)}} (-i \tau)^\frac{r}{2}
	\sum_{\bm \ell \in A^{-1}\Z^r / \Z^r} e^{\frac{2\pi i}{\tau}Q(\bm z + \bm \ell \tau)}
	\Theta_A (\bm z + \bm \ell \tau; \tau).
\end{align*}
\item We have
\begin{equation*}\label{2.4}
	\Theta_A(-\bm{z};\tau) = -\Theta_A(\bm{z};\tau).
\end{equation*}
\item For $\bm n \in \Z^2$, $\bm m \in A^{-1}\Z^2$, we have
\begin{equation*}\label{elliptic}
	\Theta_A (\bm z + \bm n \tau + \bm m ; \tau) = e^{-2\pi i B(\bm n , \bm z)} q^{-Q(\bm n)} \Theta_A (\bm z ; \tau).
\end{equation*}
\end{enumerate}
\rm

We finish this subsection by defining quantum modular forms, following Zagier \cite{ZaQ}.
\begin{definition*}
	A function $f:\mathcal Q\to \C$ (here $\mathcal{Q} \subseteq \mathbb{Q}$) is called a {\it quantum modular form} of weight $k \in \frac12\Z$ and multiplier $\chi$ for a subgroup $\Gamma$ of $\SL_2(\Z)$ and quantum set $\mathcal Q$, if for $M=\left(\begin{smallmatrix}
		a & b\\ c & d
	\end{smallmatrix}\right)\in\Gamma$ the function
	\begin{equation}\label{ob}
		f(\tau)-\chi(M)^{-1}(c\tau+d)^{-k}f(M\tau)
	\end{equation}
	can be extended to an open subset of $\R$ and is real-analytic there.
\end{definition*}

\begin{remark}
	 Zagier \cite{Za20} recently also defined {\it holomorphic quantum modular forms}. These are {\it holomorphic functions} $f:\HH\to\C$, such that \eqref{ob} is holomorphic in a larger domain than $\HH$.
\end{remark}
An example of a holomorphic quantum modular form is the generating function for the number of divisors. We have the following by \cite[Theorem 1]{BC} (see also \cite{NR}).
\begin{lemma}\label{hq}
The function
	\begin{align*}
	D(q):=	\sum_{n \geq 1} \sum_{d|n} q^n = \sum_{n\ge1} \frac{q^n}{1-q^n}
	\end{align*}
is a holomorphic quantum modular form of weight one.
\end{lemma}

\subsection{$q$-series results}

We finish this section by recalling several $q$-series identities needed in the paper. 
We often use Euler's identity 
\begin{equation}\label{109}
	\frac{1}{(\zeta)_\infty}=\sum_{n \geq 0} \frac{\z^n}{(q)_n}.
\end{equation}

We also require Bailey's pairs \cite[Chapter 3]{AndrewsBook}.
Recall that a pair of sequences $(\alpha_n,\beta_n)$ is called a {\it Bailey pair} relative to 
$(a,q)$ if
\begin{align*}
	\beta_n&=\sum_{0\le j\le n}\frac{\alpha_j}{(q)_{n-j}(aq)_{n+j}}.
\end{align*}
Bailey's Lemma is as follows.
\begin{lemma}\cite[Theorem 3.4]{AndrewsBook}\label{L:BaileysLemma}
	If $(\alpha_n,\beta_n)$ is a Bailey pair relative to 
	$(a,q)$, then (assuming convergence conditions) we have 
	\begin{align*}
		\sum_{n\ge0}(\varrho_1)_n(\varrho_2)_n\left(\frac{aq}{\varrho_1\varrho_2}\right)^n \beta_n
		&=
		\frac{ \left( \frac{aq}{\varrho_1}\right)_\infty \left(\frac{aq}{\varrho_2}\right)_\infty}
		{\left( aq\right)_\infty \left( \frac{aq}{\varrho_1\varrho_2}\right)_\infty}
		\sum_{n\ge0}\frac{ (\varrho_1)_n(\varrho_2)_n\left(\frac{aq}{\varrho_1\varrho_2}\right)^n}
		{\left( \frac{aq}{\varrho_1}\right)_n \left( \frac{aq}{\varrho_2}\right)_n }\alpha_n
		.
	\end{align*}
\end{lemma}

Recall an identity by Andrews and Freitas \cite[Corollary 4.3]{AF}: 
\begin{align} \label{id1}
\frac{1}{(q)_\infty} \sum_{n \geq 0} \zeta^n \left( (q)_n -(q)_\infty) \right) =\sum_{n \geq 1} \frac{q^n}{\left(1-\zeta q^n\right) (q)_n},
\end{align}
and another identity \cite[formula (5.1)]{Gupta}
\begin{align} \label{id2}
& \sum_{n \geq 0} \zeta^n \left( (q)_n -(q)_\infty \right) = \sum_{n \geq 1} q^n \left(1+\zeta+ \cdots + \zeta^{n-1}\right) (q)_{n-1}.
\end{align}
If $\zeta=1$, this recovers Zagier's identity \cite[formula (16)]{Za}, and for $\zeta=0$ we get 
\begin{equation}\label{tail}
\sum_{n \geq 0} q^{n+1} (q)_n = 1-(q)_\infty.
\end{equation}
We also require an identity of Andrews, Garvan, and Liang \cite[Theorem 3.5]{AGL}
\begin{equation} \label{AGL}
	\sum_{n \geq 0} \zeta^{n} \left(\left(q^{n+1}\right)_\infty -1 \right)=\sum_{n \geq 1} \frac{(-1)^n q^{\frac{n(n+1)}{2}}}{\left(1-\zeta q^n\right) (q)_n}.
\end{equation}
Finally, we require two of Fine's identities \cite[equations (12.42),(12.45)]{Fine}
\begin{align} \label{fine}
 \sum_{n \geq 1} \frac{(-1)^{n+1} q^{\frac{n^2+n}{2}}}{\left(1-q^n\right)(q)_n}&=D(q), \\ \label{fine2}
 \sum_{n \geq 0} \left(\frac{1}{(q)_\infty}-\frac{1}{(q)_n} \right)&=\frac{1}{(q)_\infty} D(q).
\end{align}

\section{$A$-series and the proof of Theorem \ref{A7.A8}}

We start our investigation of mock and quantum modular properties of graph series by focusing on the path graphs (i.e., Dynkin diagrams of type $A$) denoted by 
$A_k$, $k \geq 1$ (as usual $A_1$ is just a single node):
\begin{center}
$$1 -2- 3 - \cdots - k$$
\end{center}
The corresponding graph series are given by 
$$
H_{A_k}(q)=\sum_{\bm n \in \N_0^k} \frac{q^{n_1 n_2 + \cdots + n_{k-1} n_k+n_1+ \cdots + n_k}}{(q)_{n_1} \cdots (q)_{n_k}}.
$$
Using \eqref{109} on the sums for $n_1$, $n_2$, $n_{k+3}$, and $n_{k+4}$ and relabeling, it is easy to see that 
for $k \geq 3$,
\begin{equation} \label{shifting}
H_{A_{k+4}}(q)=\frac{1}{(q)_\infty^2} \mathcal{H}_{A_{k}}(q),
\end{equation}
where for $k\ge2$
$$
\mathcal{H}_{A_{k}}(q):=\sum_{\bm n \in \N_0^k} \frac{q^{n_1 n_2+\cdots +n_{k-1}n_k +n_1+\cdots +n_{k}}}{(q)_{n_1+1}(q)_{n_2} \cdots (q)_{n_{k-1}}(q)_{n_{k}+1}}.
$$
Further applications of Euler's formula \eqref{109} give
$$
\mathcal{H}_{A_k}(q)= q^{-2} \sum_{n_1,...,n_{k-2} \geq 0}  \left(\frac{1}{(q)_\infty} -\frac{1}{(q)_{n_{k-2}}} \right) \left(\frac{1}{(q)_\infty}-\frac{1}{(q)_{n_1}} \right) \frac{q^{n_1n_2+...+n_{k-3}n_{k-2}+n_2+\cdots +n_{k-3}}}{(q)_{n_2} \cdots (q)_{n_{k-3}}}.
$$

The following identities are taken from \cite{JM}: 
\begin{align}
H_{A_1}(q)&=\frac{1}{(q)_\infty}, \quad
H_{A_2}(q) =\frac{1}{(1-q)(q)_\infty},  \quad
H_{A_3}(q)=\frac{q^{-1}\left(1-(q)_\infty \right)}{(q)_\infty^2}, \quad 
H_{A_4}(q)=\frac{q^{-1}}{(q)_\infty^2} D(q), 
\nonumber \\
H_{A_5}(q)&=\frac{q^{-1}}{(q)_\infty^3} \sum_{n \geq 0}\lp (q)_n -(q)_\infty \rp,\quad
H_{A_6}(q)=
	\frac{2q^{-1}}{(q)_\infty^3} D(q)
	-
	\frac{q^{-1}}{(q)_\infty^3}
	+
	\frac{q^{-1}}{(q)_\infty^2}. 
\nonumber	
\end{align}
We are now ready to prove Theorem \ref{A7.A8}.

\begin{proof}[Proof of Theorem \ref{A7.A8}]
For the first identity, we compute 
\[
H_{A_7}(q) = \frac{1}{(q)_\infty^2} \mathcal{H}_{A_3}(q)
= \frac{1}{(q)^3_\infty} \sum_{n\ge0} q^{-n-2} \left(\frac{1}{(q)_\infty}-\frac{1}{(q)_{n}}\right) \biggl((q)_{n}-(q)_{\infty} \biggr).
\]
To analyze the last sum, we introduce a new parameter $\z$ and consider
\begin{align*}
F_\z(q) &:= \sum_{n\ge0} \z^{n} \left( \frac{1}{(q)_\infty}-\frac{1}{(q)_{n}} \right)  \biggl( (q)_{n}-(q)_{\infty} \biggr)
= \sum_{n\ge0} \z^{n} \left(\left(q^{n+1}\right)_\infty  + \frac{1}{\left(q^{n+1}\right)_\infty}-2 \right)\\
&= \sum_{n\ge0} \zeta^{n} \left(\left(q^{n+1}\right)_\infty -1 \right)+\sum_{n \geq 0} \z^{n} \left( \frac{1}{\left(q^{n+1}\right)_\infty}-1 \right).
\end{align*}
Adding \eqref{AGL} and \eqref{id1} results in cancellation of the term $n=1$, so we have
$$
F_\z(q) = \sum_{n\ge2} \frac{(-1)^n q^{\frac{n(n+1)}{2}}+q^n}{\left(1-\z q^n\right) (q)_n}.
$$
Letting $\z = q^{-1}$ (which is now allowed) gives 
$$
F_{q^{-1}}(q) = \sum_{n \geq 0} q^{-n} \left( \frac{1}{(q)_\infty}-\frac{1}{(q)_{n}} \right)  \biggl( (q)_{n}-(q)_{\infty} \biggr)=\sum_{n \geq 2} \frac{(-1)^n q^{\frac{n(n+1)}{2}}+q^n}{\left(1-q^{n-1}\right) (q)_n}.
$$
Next we split the right-hand side into two sums. For the first sum we recall (\ref{fine}) and also observe 
\begin{equation*}
\frac{1}{1-q} \sum_{n \geq 1} \frac{(-1)^{n+1} q^{\frac{n^2+n}{2}+1}}{\left(1-q^n\right) (q)_n}-\frac{q^2}{(1-q)^2}=\sum_{n \geq 2} \frac{(-1)^n q^{\frac{n^2+n}{2}}}{\left(1-q^{n-1}\right)(q)_n},
\end{equation*}
which follows from the ``finite'' identity (here $k \geq 2$)
\begin{equation*}
\left( {\frac{1}{1-q} \sum_{n = 1}^{k-1}  \frac{(-1)^{n+1} q^{\frac{n^2+n}{2}+1}}{\left(1-q^n\right) (q)_n}-\frac{q^2}{(1-q)^2}} \right)- \left(\sum_{n= 2}^{k} \frac{(-1)^n q^{\frac{n^2+n}{2}}}{\left(1-q^{n-1}\right)(q)_n} \right)=\frac{q^{\frac{k^2}{2}+\frac{k}{2}+1}}{(1-q)(q)_k},
\end{equation*}
after letting $k \to \infty$. Combined this implies that
\begin{equation}\label{EE}
\sum_{n \geq 2} \frac{(-1)^n q^{\frac{n(n+1)}{2}}}{(1-q^{n-1}) (q)_n}=\frac{q}{1-q} \sum_{n \geq 2} \frac{q^n}{1-q^n}.
\end{equation}

For the second sum we reintroduce the parameter $\z$ and use that
$$
\sum_{n \geq 2} \frac{q^{n}}{(1-\z q^{n})(q)_n}= \sum_{n \geq 1} \frac{q^n}{(1-\z q^n) (q)_n} - \frac{q}{(1-q)(1-\z q)}.
$$
Employing \eqref{id2} and 
$$
\frac{q}{(1-q)(1-\z q)}=\sum_{n \geq 1} \left(1+ \cdots + \z^{n-1}\right) q^n
$$
we get 
\begin{equation}\label{id3}
\sum_{n \geq 2} \frac{q^{n}}{(1-\z q^{n})(q)_n} = \sum_{n \geq 1} q^n \left(1+\z + \cdots + \z^{n-1}\right) \left(\frac{1}{(q^n)_\infty}-1 \right).
\end{equation}
We let $\z=q^{-1}$ in \eqref{id3} to obtain 
\begin{align*}
\sum_{n \geq 2} \frac{q^{n}}{(1-q^{n-1})(q)_n} &= \sum_{n \geq 1} q^n \left(1+q^{-1} + \cdots + q^{-n+1}\right) \left(\frac{1}{(q^n)_\infty}-1 \right)\\
&= \sum_{n \geq 1} q^n \left(q^n+q^{n-1} + \cdots + q\right) \left(\frac{1}{(q^n)_\infty}-1 \right) = \frac{q}{1-q} \sum_{n \geq 1} (1-q^{n})\left(\frac{1}{(q^n)_\infty}- 1\right)\\
&= -\frac{q}{1-q} \sum_{n \geq 1} q^{n} \left(\frac{1}{(q^n)_\infty}- 1\right)+ \frac{q}{1-q} \sum_{n \geq 1} \left(\frac{1}{(q^n)_\infty}- 1\right).
\end{align*}
The first sum evaluates as
\begin{multline*}
-\frac{q^2}{1-q} \sum_{n \geq 1} q^{n-1} \left(\frac{1}{(q^n)_\infty}- 1\right) = -\frac{q^2}{1-q} \sum_{n \geq 0} q^{n} \left(\frac{1}{(q^{n+1})_\infty}- 1\right)\\
= -\frac{q^2}{(1-q)(q)_\infty} \sum_{n\ge0} q^n ((q)_n-(q)_\infty)
= -\frac{q}{(1-q)(q)_\infty} + \frac{q}{1-q}+\frac{q^2}{(1-q)^2},
\end{multline*}
where we use \eqref{id1} with $\zeta=q$ and \eqref{109}.
Thus
\begin{align*}
	\sum_{n \geq 2} \frac{q^n}{\lp 1-q^{n-1}\rp (q)_n}
	&=-\frac{q}{(1-q)(q)_\infty} + \frac{q}{1-q}+\frac{q^2}{(1-q)^2} 
	+ \frac{q}{1-q} \sum_{n \geq 1} \left(\frac{1}{(q^n)_\infty}- 1\right).
\end{align*}
 Combined with (\ref{EE}), the previous relation gives 
\begin{align*}
&\sum_{n \geq 2} \frac{(-1)^n q^{\frac{n(n+1)}{2}}+q^n}{(1-q^{n-1}) (q)_n}\\
&\hspace{3cm}= \frac{q}{1-q} \sum_{n\ge2} \frac{q^n}{1-q^n}-\frac{q}{(1-q)(q)_\infty} +\frac{q}{1-q}+\frac{q^2}{(1-q)^2}+ \frac{q}{1-q} \sum_{n \geq 1} \left(\frac{1}{(q^n)_\infty}- 1\right)\\
&\hspace{3cm}= \frac{q}{1-q} \left(D(q) -\frac{1}{(q)_\infty} +1 +\frac{1}{(q)_\infty} G(q) \right).
\end{align*}
Finally, after we multiply by $\frac{q^{-2}}{(q)_\infty^3}$ we get the claimed formula.

For $H_{A_8}(q)$, using (\ref{shifting}), we first get
\begin{align*}
H_{A_{8}}(q)&=\frac{q^{-2}}{(q)_\infty^2}  \sum_{n_1, n_2 \geq 0}  \left(\frac{1}{(q)_\infty} -\frac{1}{(q)_{n_1}} \right) \left(\frac{1}{(q)_\infty}-\frac{1}{(q)_{n_2}} \right)q^{n_1n_2}
\\ 
&=\frac{q^{-2}}{(q)_\infty^2}  \sum_{n_1, n_2 \geq 0}  \left(\frac{q^{n_1 n_2}}{(q)^2_\infty} -\frac{q^{n_1 n_2}}{(q)_{n_1} (q)_\infty} -\frac{q^{n_1 n_2}}{(q)_{n_2} (q)_\infty}+
\frac{q^{n_1 n_2}}{(q)_{n_1}(q)_{n_2}} \right).
\end{align*}
We would like to separate this into four sums but there is a convergence issue. For this reason, we first evaluate the terms with $n_1n_2=0$.
For $n_1=n_2=0$ we have 
$$
\frac{q^{-2}}{(q)_\infty^2}  \left( \frac{1}{(q)^2_\infty} -\frac{2}{(q)_\infty}+1 \right).
$$
For $n_2 \geq 1$, $n_1=0$ and $n_1 \geq 1$, $n_2=0$ (due to symmetry) we get the contribution 
$$
-\frac{2 q^{-2}}{ (q)_\infty^2}  \sum_{n \geq 1} \left(-1+\frac{1}{(q)_\infty} \right) \left(\frac{1}{(q)_\infty}-\frac{1}{(q)_n}\right).
$$
We are left with 
\begin{multline*}
\sum_{n_1,n_2\ge1}  \left(\frac{q^{n_1 n_2}}{(q)^2_\infty} -\frac{q^{n_1 n_2}}{(q)_{n_1} (q)_\infty} -\frac{q^{n_1 n_2}}{(q)_{n_2} (q)_\infty} + \frac{q^{n_1 n_2}}{(q)_{n_1}(q)_{n_2}}\right)\\
= \frac{D(q)}{(q)_\infty^2}-\frac{2}{(q)_\infty} \sum_{n \geq 0} \left(\frac{1}{(q^{n+1})_\infty}-1 \right)+\sum_{n_1,n_2 \geq 1} \frac{q^{n_1 n_2}}{(q)_{n_1}(q)_{n_2}}.
\end{multline*}
For the final sum we use an identity from \cite[Section 7.3]{JM} (which is essentially (\ref{fine2})):
$$
\sum_{n_1,n_2 \geq 1} \frac{q^{n_1 n_2}}{(q)_{n_1}(q)_{n_2}}=1+2 \frac{D(q)}{(q)_\infty}-\frac{1}{(q)_\infty}.
$$
Combining with the above, and using (\ref{fine}), yields the claim.
\end{proof}

\begin{remark} As discussed in \cite{JM}, Theorem \ref{A7.A8} implies a bosonic formula  
$$H_{A_7}(q)=\frac{q^{-1}}{(1-q)(q)_\infty^4} \left( \sum_{n \geq 1}(-1)^n (-3n+1) q^{\frac{3n^2+n}{2}} 
+ \sum_{n \leq -1}(-1)^n (3n+2) q^{\frac{3n^2+n}{2}} \right).$$
Note however that the conjecture for $H_{A_8}(q)$ given in \cite{JM} does not hold.
\end{remark}

The following result describes the asymptotic behaviors and quantum modular properties of these graph series. 
\begin{proposition} \label{seriesA} As $t \to 0^+$, we have:
\begin{itemize}[leftmargin=0.7cm]
\item[\rm(1)] $(e^{-t})_\infty H_{A_2}(e^{-t}) = \frac{1}{t} +O(1)$,
\item[\rm(2)] $(e^{-t})_\infty^2 H_{A_3}(e^{-t}) = 1 + O(t)$, 
\item[\rm(3)] $(e^{-t})_\infty^2 H_{A_4}(e^{-t}) = \frac{ \gamma-\log(t)}{t}+O(1),$
where $\gamma$ is the Euler--Mascheroni constant,
\item[\rm(4)] $(e^{-t})^{3}_\infty H_{A_5}(e^{-t}) = 1 +O(t)$,
\item[\rm(5)] $ (e^{-t})^{3}_\infty H_{A_6}(e^{-t}) = \frac{2 (\gamma-\log(t))}{t}+O(1)$, 
\item[\rm(6)] $ (e^{-t})^{4}_\infty H_{A_7}(e^{-t}) = 1+O(t)$,
\item[\rm(7)] $ (e^{-t})^{4}_\infty H_{A_8}(e^{-t}) = \frac{3 (\gamma-\log(t))}{t}+O(1).$

\end{itemize}
Moreover, $q (q)_\infty^2 H_{A_4}(q)$ is a holomorphic quantum modular form of weight one, while $q (q)^{3}_\infty H_{A_5}(q)$ and $q (1-q) (q)^{3}_\infty H_{A_7}(q)$ are quantum modular forms of weight $\frac32$.
\end{proposition}

\begin{proof}
Statements (1) and (2) are immediate. The asymptotic behavior in (3) is well-known and can easily be concluded using the Euler--Maclaurin summation formula \cite{Za2006}. Quantum modular properties of $D(q)$ are given in Lemma \ref{hq}. 
For (4), we rewrite (see Theorem 2 of \cite{Za}):
\begin{equation} \label{z-identity}
G(q)=-\frac 12 H(q)+(q)_\infty \lp \frac 12 - D(q) \rp,
\end{equation}
where 
$$
H(q):=\sum_{n \geq 1} n \left(\frac{12}{n}\right) q^{\frac{n^2-1}{24}}
$$
is quantum modular of weight $\frac32$ and satisfies the asymptotic behavior \cite[Theorem 3]{Za}
$$
H\left(e^{-t}\right) = -2 -2t +O\left(t^2\right).
$$
Thus
\[
 \lp e^{-t} \rp^3_\infty H_{A_5}\left(e^{-t}\right) = -\frac 12 H \lp e^{-t} \rp  \lp 1+ O(t) \rp= 1+O(t).
\]

To see (5), we use part (3). For (6) we recall an identity from \cite[Section 7.4]{JM} 
\begin{multline*}
 \sum_{n \geq 1} (-3n+1)(-1)^n q^{\frac{3n^2+n}{2}} 
+ \sum_{n \leq -1} (3n+2)(-1)^n q^{\frac{3n^2+n}{2}} 
 =-1+(q)_\infty D(q)+G(q) +(q)_\infty.
\end{multline*}
Using (\ref{z-identity}) we get 
$$
-1+G\lp e^{-t} \rp = t +O\left(t^2\right).
$$
Since $\frac{1}{1-e^{-t}} = \frac{1}{t} + O(t)$, (6) follows.
Quantum modularity follows as before. For (7) we use exactly the same argument as in (5).
\end{proof}
Based on Proposition \ref{seriesA}, we conjecture that the following is true.
\begin{conjecture}
For $n \geq 1$, there exist $a_n,b_n,c_n\in\R^+$ such that
$$
\left(e^{-t}\right)_\infty^n H_{A_{2n}}\left(e^{-t}\right)  \sim \frac{a_n+c_n \log(t)}{t}, \qquad
\left(e^{-t}\right)_\infty^n H_{A_{2n-1}}\left(e^{-t}\right)  \sim b_n,
\qquad\text{as } t\rightarrow0^{+}.
$$
\end{conjecture}

\section{$5$-cycles, sums of divisors, and the proof of \eqref{HC5} and \eqref{g8}}

In this part we are concerned with series coming from certain graphs obtained by glueing $5$-cycles. 
Generally, graph series associated to graphs with cycles are more complicated to analyze. We start from an auxiliary result that, quite surprisingly, allows us to perform computations for several interesting examples of graphs. The next lemma can be viewed as a generalization of the $A_2$-identity discussed in the previous section.
\begin{lemma} \label{framing}
For $a,b \in \mathbb{N}_0$, we have
\begin{align*} A(a,b) :&=\sum_{n_1,n_2 \geq 0} \frac{q^{n_1 n_2+(a+1)n_1+(b+1)n_2}}{(q)_{n_1}(q)_{n_2}} 
=\frac{1}{ \left(q^{b+1}\right)_{a+1}\left(q^{a+1}\right)_\infty}=\frac{1}{\left(q^{a+1}\right)_{b+1}\left(q^{b+1}\right)_\infty}.
\end{align*}
\end{lemma}
\begin{proof} The second equality follows due to the symmetry $A(a,b)=A(b,a)$. To show the first, we recall a well-known formula (see \cite[equation (6.2)]{Fine})
\begin{equation} \label{sss}
\sum_{n \geq 0} \frac{(s q)_n}{(q)_n} t^n=\frac{(s tq)_\infty}{(t)_\infty}.
\end{equation}
We compute
\begin{align*} 
A(a,b) =\sum_{n_2 \geq 0} \frac{q^{(b+1)n_2}}{\left(q^{n_2+a+1}\right)_\infty (q)_{n_2}}
=\frac{1}{\left(q^{a+1}\right)_\infty} \sum_{n \geq 0}\frac{\lp q^{a+1}\rp_{n} q^{(b+1)n}}{(q)_{n}} = \frac{1}{\left(q^{a+1}\right)_\infty \left(q^{b+1}\right)_{a+1} },
\end{align*}
where the last equality follows from (\ref{sss}), letting $s=q^{a+1}$ and $t=q^{b+1}$.
\end{proof}
Equipped with this result we can now give several elegant representations of graph series associated to a $5$-cycle

\begin{center}
\begin{tikzpicture}
  \node (n4) at (4.3,3)  {4};
  \node (n5) at (4,4)  {5};
  \node (n1) at (5,5) {1};
  \node (n2) at (6,4)  {2};
  \node (n3) at (5.6,3)  {3};
    \foreach \from/\to in {n4/n5,n5/n1,n1/n2,n2/n3,n3/n4}
    \draw (\from) -- (\to);
    
  \node at (6.5,3.9)  {,};
\end{tikzpicture}
\end{center}
in particular proving \eqref{HC5}.

\begin{proposition}\label{prop:4.2} We have
\begin{align*} 
\sum_{\bm n \in \N_0^5}  \frac{q^{n_1 n_2+n_1 n_5+n_2 n_3+n_3n_4+n_4 n_5+n_1+n_2+n_3+n_4+n_5}}{(q)_{n_1}(q)_{n_2}(q)_{n_3}(q)_{n_4}(q)_{n_5}}&=\frac{q^{-1}}{(q)^2_\infty} \sum_{n \geq 1} \frac{nq^n}{1-q^n},\\
\sum_{\bm n \in \N_0^5}\frac{q^{n_1 n_2+n_1 n_5+n_2 n_3+n_3n_4+n_4 n_5+2n_1+n_2+n_3+n_4+n_5}}{(q)_{n_1}(q)_{n_2}(q)_{n_3}(q)_{n_4}(q)_{n_5}}&={\frac{1}{(1-q)^2(q)^2_\infty}},\\
\sum_{\bm n \in \N_0^5}\frac{q^{n_1 n_2+n_1 n_5+n_2 n_3+n_3n_4+n_4 n_5+n_1+2n_2+n_3+n_4+2n_5}}{(q)_{n_1}(q)_{n_2}(q)_{n_3}(q)_{n_4}(q)_{n_5}}&={\frac{q^{-2}}{(q)^2_\infty} \sum_{n \geq 2} \frac{nq^n}{1-q^n}}.
\end{align*}
\end{proposition} 
\begin{proof}
Let 
$$
B_1(q):=\sum_{n \geq 1} \frac{nq^n}{1-q^n}, \qquad B_2(q):={\frac{1}{(1-q)^2(q)^2_\infty}}, \qquad B_3(q):=\frac{q^{-2}}{(q)^2_\infty} \sum_{n \geq 2} \frac{nq^n}{1-q^n}.
$$
Note that  $B_2(q)+qB_3(q)=B_1(q)$. It is easy to see that the same relation holds for the left-hand sides.  Thus it suffices to prove the first two identities. We start with the second identity.  Euler's identity \eqref{109} gives 
\begin{align*}
& \sum_{\bm n \in \N_0^5}\frac{q^{n_1 n_2+n_1 n_5+n_2 n_3+n_3n_4+n_4 n_5+2n_1+n_2+n_3+n_4+n_5}}{(q)_{n_1}(q)_{n_2}(q)_{n_3}(q)_{n_4}(q)_{n_5}}=\sum_{n_2,n_3,n_4,n_5 \geq 0} \frac{q^{n_2 n_3+n_3 n_4+n_4 n_5+n_2+n_3+n_4+n_5}}{\left(q^{n_2+n_5+2}\right)_\infty(q)_{n_2}(q)_{n_3}(q)_{n_4}(q)_{n_5}} \\
& = \frac{1}{(q)_\infty } \sum_{n_2,n_3,n_4,n_5 \geq 0} \frac{q^{n_2 n_3+n_3 n_4+n_4 n_5+n_2+n_3+n_4+n_5} (q)_{n_2+n_5+1} }{(q)_{n_2}(q)_{n_3}(q)_{n_4}(q)_{n_5}} \\
& =\frac{1}{(q)_\infty} \sum_{n_2,n_5 \geq 0} \frac{q^{n_2+n_5} (q)_{n_2+n_5+1}}{(q)_{n_2}(q)_{n_5}} {\sum_{n_3,n_4 \geq 0} \frac{q^{n_3 n_4+(n_2+1)n_3+(n_5+1)n_4}}{(q)_{n_3}(q)_{n_4}}}.
\end{align*}
Using this equals
\begin{align*}
 & \frac{1}{(q)_\infty} \sum_{n_2 , n_5 \geq 0} \frac{q^{n_2+n_5} (q)_{n_2+n_5+1}}{(q)_{n_2}(q)_{n_5}} \frac{1}{\left(q^{n_2+1}\right)_\infty \left(q^{n_5+1}\right)_{n_2+1}}= \frac{1}{(q)_\infty^2} \sum_{n_2,n_5 \geq 0} q^{n_2+n_5}=\frac{1}{(1-q)^2 (q)_\infty^2 }.
\end{align*}

For the first identity we use a similar argument. We get
\begin{align*}
& \sum_{\bm n \in \N_0^5} \frac{q^{n_1 n_2+n_1 n_5+n_2 n_3+n_3n_4+n_4 n_5+n_1+n_2+n_3+n_4+n_5}}{(q)_{n_1}(q)_{n_2}(q)_{n_3}(q)_{n_4}(q)_{n_5}} =\sum_{n_2,n_3,n_4,n_5 \geq 0} \frac{q^{n_2 n_3+n_3 n_4+n_4 n_5+n_2+n_3+n_4+n_5}}{\left(q^{n_2+n_5+2}\right)_\infty(q)_{n_2}(q)_{n_3}(q)_{n_4}(q)_{n_5}}  \\
&= \frac{1}{(q)_\infty}  \sum_{n_2,n_5 \geq 0} \frac{q^{n_2+n_5}}{1-q^{n_2+n_5+1}}=\frac{q^{-1}}{(q)_\infty^2} \sum_{n \geq 1} \frac{nq^n}{1-q^n},
\end{align*}
as claimed.
\end{proof}
Next we consider the graph series associated to the graph 
\begin{center}
\begin{tikzpicture}
  \node (n4) at (4.3,3)  {4};
  \node (n5) at (4,4)  {5};
  \node (n1) at (5,5) {1};
  \node (n2) at (6,4)  {2};
  \node (n3) at (5.6,3)  {3};
   \node (n6) at (9,4)  {6};
     \node (n7) at (7,4)  {7};
  \node (n8) at (8,4)  {8};
  \foreach \from/\to in {n4/n5,n5/n1,n1/n2,n2/n3,n3/n4,n2/n7,n7/n8,n8/n6,n3/n6,n1/n6}
    \draw (\from) -- (\to);
\end{tikzpicture}
\captionof{figure}{Graph $\Gamma_8$}
\end{center}

We need a result for the sum of squares of divisors. 
\begin{lemma}\label{bell}
We have
$$\sum_{\bm n \in \N_0^3} \frac{q^{n_1+n_2+n_3+1}}{(1-q^{n_1+n_2+1})(1-q^{n_1+n_2+n_3+1})}=\sum_{n \geq 1} \frac{n^2 q^n}{1-q^n}.
$$
\end{lemma}
\begin{proof}  We can write the left-hand side as 
\begin{align*}
\sum_{0 \leq \ell \leq k \atop k \geq 0} \frac{(\ell+1)q^{k+1}}{(1-q^{\ell+1})(1-q^{k+1})}
&=\sum_{1 \leq \ell \leq k \atop k \geq 1} \frac{\ell q^{k}}{(1-q^{\ell})(1-q^{k})} 
=\sum_{1 \leq \ell \leq k \atop k \geq 1} \frac{\ell q^{k}\lp q^\ell+1-q^\ell\rp}{\left(1-q^{k}\right)\left(1-q^{\ell}\right)}
\\
&=\sum_{1 \leq \ell \leq k \atop k \geq 1} \frac{\ell q^{k+\ell}}{(1-q^{\ell})(1-q^{k})}
+\frac 12\sum_{k \geq 1} \frac{k(k+1)q^k}{1-q^k}
\\
&=\frac12 \sum_{k, \ell \geq 1} \frac{{\rm min}(k,\ell)q^{k+\ell}}{(1-q^k)(1-q^\ell)}+\frac12 \sum_{k \geq 1} \frac{kq^{2k}}{(1-q^k)^2}+\frac 12\sum_{k \geq 1} \frac{k(k+1)q^k}{1-q^k}.
\end{align*}
Finally plugging in \cite[equations (5.4) and (6.5)]{stacks}
$$ \sum_{k, \ell \geq 1 } \frac{ {\rm min}(k,\ell) q^{k+\ell}}{(1-q^{k})(1-q^{\ell})}
=\sum_{n \geq 1} \frac{n(n-1)q^n}{1-q^n}-\sum_{n \geq 1} \frac{nq^{2n}}{(1-q^n)^2},$$
gives the claim.
\end{proof}

\begin{remark} 
Recall Bell's identity for the sum of squares of divisors \cite[equation (2.3)]{stacks}
$$
\sum_{n\ge1} \frac{q^n}{(1-q^n)^2} \left(\frac{1}{1-q} + \frac{1}{1-q^2} + \cdots + \frac{1}{1-q^n}\right)=\sum_{n \geq 1} \frac{n^2 q^n}{1-q^n}.
$$
Curiously, in Lemma \ref{bell} we prove a slightly different identity
$$\sum_{n\ge1} \frac{q^n}{1-q^n} \left(\frac{1}{1-q} +\frac{2}{1-q^2}+ \cdots + \frac{n}{1-q^n}\right)=\sum_{n \geq 1} \frac{n^2 q^n}{1-q^n}.$$
\end{remark}

Now we are ready to  prove \eqref{g8}.
\begin{proof}[Proof of (\ref{g8})]
We enumerate the vertices of $\Gamma_8$ as on Figure 1. 
We use \eqref{109} for $n_1$ to obtain
\begin{align*}
& \sum_{\bm{n}\in\N_0^8} \frac{q^{n_1n_2+n_1n_5+n_1n_6+n_2n_3+n_2n_7+n_3n_4+n_3n_6+n_4n_5+n_6n_8+n_7n_8+n_1+n_2+n_3+n_4+n_5+n_6+n_7+n_8}}
{(q)_{n_1}(q)_{n_2}(q)_{n_3}(q)_{n_4}(q)_{n_5}(q)_{n_6}(q)_{n_7}(q)_{n_8}} \\
& =\sum_{n_2,n_3,n_4,n_5,n_6,n_7,n_8 \geq 0} \frac{q^{n_2n_3+n_2n_7+n_3n_4+n_3n_6+n_4n_5+n_6n_8+n_7n_8+n_2+n_3+n_4+n_5+n_6+n_7+n_8}}
{(q^{n_2+n_5+n_6+1})_\infty (q)_{n_2}(q)_{n_3}(q)_{n_4}(q)_{n_5}(q)_{n_6}(q)_{n_7}(q)_{n_8}} \\
& =\frac{1}{(q)_\infty} \sum_{n_2,n_5,n_6,n_7,n_8 \geq 0} \frac{(q)_{n_2+n_5+n_6} q^{n_2n_7+n_6n_8+n_7n_8+n_2+n_5+n_6+n_7+n_8}}
{ (q)_{n_2} (q)_{n_5}(q)_{n_6}  (q)_{n_7}(q)_{n_8} } 
\\
	& \hspace{8cm}\times \sum_{n_3,n_4\ge0} \frac{q^{n_3n_4+(n_2+n_6+1)n_3+(n_5+1)n_4}}{(q)_{n_3}(q)_{n_4}} \\
& =  \frac{1}{(q)_\infty} \sum_{n_2,n_5,n_6,n_7,n_8 \geq 0} \frac{(q)_{n_2+n_5+n_6} q^{n_2n_7+n_6n_8+n_7n_8+n_2+n_5+n_6+n_7+n_8}}
{ (q^{n_5+1})_\infty (q^{n_2+n_6+1})_{n_5+1} (q)_{n_2} (q)_{n_5}(q)_{n_6} (q)_{n_7}(q)_{n_8}}\\
&=  \frac{1}{(q)^2_\infty} \sum_{n_2,n_5,n_6,n_7,n_8 \geq 0} \frac{(q)_{n_2+n_5+n_6} q^{n_2n_7+n_6n_8+n_7n_8+n_2+n_5+n_6+n_7+n_8}}
{ (q^{n_2+n_6+1})_{n_5+1} (q)_{n_2} (q)_{n_6} (q)_{n_7}(q)_{n_8}}
\end{align*}
using Lemma \ref{framing} in the penultimate step. Using \eqref{109}, we then rewrite this as 
\begin{align*}
& \frac{1}{(q)^2_\infty}  \sum_{n_2,n_5,n_6 \geq 0} \frac{(q)_{n_2+n_5+n_6} q^{n_2+n_5+n_6}}
{ (q)_{n_2} (q)_{n_6}  (q^{n_2+n_6+1})_{n_5+1} (q^{n_2+1})_\infty (q^{n_6+1})_{n_2+1}} \\
&=\frac{1}{(q)^3_\infty} \sum_{n_2,n_5,n_6 \geq 0} \frac{(q)_{n_2+n_5+n_6} q^{n_2+n_5+n_6}}
{ (q)_{n_6}  (q^{n_2+n_6+1})_{n_5+1} (q^{n_6+1})_{n_2+1}} \\
&=\frac{q^{-1}}{(q)^3_\infty} \sum_{n_2,n_5,n_6 \geq 0} \frac{ q^{n_2+n_5+n_6+1}}
{(1-q^{n_2+n_6+1})(1-q^{n_2+n_5+n_6+1})}
=\frac{q^{-1}}{(q)^3_\infty} \sum_{n \geq 1} \frac{n^2 q^n}{1-q^n},
\end{align*}
where the last equality is due to Lemma \ref{bell}.
\end{proof}



\section{$D_4$ graph series and the proof of \eqref{HD4}}

In this part we investigate graph series associated to the graph of type $D_4$:
$$
H_{D_4}(q) = \sum_{\bm n \in \N_0^4} \frac{q^{n_1 n_2 + n_1 n_3+n_1 n_4+ n_1+n_2+n_3+n_4}}{(q)_{n_1} (q)_{n_2} (q)_{n_3} (q)_{n_4}}.$$

We first obtain a representation for $H_{D_4}(q)$ using Appell--Lerch sums.
Using this result we then rewrite it as an indefinite theta function of signature $(1,1)$. We also discuss mock modular properties and the asymptotic behavior as $q \to 1^-$.
To view modularity properties of $H_{D_4}$, we use 
$$
I_1(q):=
\sum_{n \in \mathbb{Z} \setminus \{0\}}
\frac{(-1)^{n+1}q^{\frac{n(3n+1)}{2}}}{(1-q^n)^2}, \qquad
I_2(q):=\sum_{n \in \mathbb{Z} \setminus \{0\}} \frac{(-1)^{n+1} n q^{\frac{n(n+1)}{2}}}{1-q^n}.
$$
Moreover, we require the generating function for so-called ranks of strongly unimodal sequences, explicitly given by $U(-\zeta;q)$, where
\begin{gather*}
	U(\zeta;q) := \sum_{n\ge0} q^{n+1} (\zeta q)_n \left(\zeta^{-1}q\right)_n.
\end{gather*}

\begin{proposition}\label{prop:4.1} We have
\begin{equation*}
H_{D_4}(q)= \frac{q^{-1}}{(q)^4_\infty} \lp I_1(q) + \frac{1}{24}(1-E_2(\t))+I_2(q)\rp
=\frac{q^{-1}}{(q)_\infty^3} U(1;q).
\end{equation*}
\end{proposition}

\begin{proof}
We use Euler's identity \eqref{109} three times (for $n_2,n_3,$ and $n_4$) to write 
\[
H_{D_4}(q)
=
	\sum_{n\ge0}
	\frac{q^{n}}
	{(q)_{n} (q^{n+1})_{\infty}^3}
=
	\frac{1}{(q)_\infty^3}	
	\sum_{n\ge0}
	q^{n}(q)_n^2 
=
	\frac{q^{-1}}{(q)_\infty^3}	
	U(1;q).
\]
To see the first identity, we use an identity of Andrews \cite[equation (1.2)]{Andrews1}, giving 
\begin{align*}
U(1;q)
&=
	\frac{1}{(q)_\infty}\left(
		\sum_{n\geq 1} 
		\frac{(-1)^{n+1}(1+q^n)q^{\frac{n(3n+1)}{2}}}{(1-q^n)^2}
		-
		\sum_{n\geq 1} \frac{q^n}{(1-q^n)^2}
		+
		2\sum_{n\geq 1} \frac{(-1)^{n+1} n q^{\frac{n(n+1)}{2}}  }
			{1-q^n}
	\right).
\end{align*}
Rewriting the first and the third sum and using that $\sum_{n \geq 1} \frac{q^n}{(1-q^n)^2}=\frac{1-E_2(\tau)}{24}$ yields the claim.
\end{proof}
We next prove the indefinite theta function representation of $H_{D_4}$ as stated in \eqref{HD4}. We note that similar formulas already exist in the literature; see for instance \cite[Theorem 1.5]{BOPR}. 
\begin{proof}[Proof of \eqref{HD4}]
	We have  
	\begin{align*}
	(q)_\infty  U(1;q) &=
		\sum_{n\geq 1} 
		\frac{(-1)^{n+1}(1+q^n)q^{\frac{n(3n+1)}{2}}}{(1-q^n)^2}
		-
		\sum_{n\geq 1} \frac{q^n}{(1-q^n)^2}
		+
		2\sum_{n\geq 1} \frac{(-1)^{n+1} n q^{\frac{n(n+1)}{2}}  }
			{1-q^n}.
\end{align*}
	We first apply Lemma \ref{AL_E2}
	to combine the first and second sum in $F(q)$ to obtain
\begin{align*}
	&(q)_\infty U(1;q)=-\sum_{n\geq 1} 
		\frac{(-1)^{n+1}\left(1+q^n\right)q^{\frac{n(n+1)}{2}}\left(1-q^{n^2}\right)}{(1-q^n)^2}+ 2 \sum_{n \geq 1} \frac{(-1)^{n+1} n q^{\frac{n(n+1)}{2}}}{1-q^n}.
		\end{align*}
		Next we use the identity:
\begin{equation*} 
-\frac{(1+q^n)\left(1-q^{n^2}\right)}{(1-q^n)^2}=-\sum_{m=0}^{n-1}(2m+1) q^{nm}- 2n \sum_{m \geq n}  q^{nm},
\end{equation*}
which follows from expanding the left-hand side as a geometric series.
Plugging this in, we have
\begin{align*}	
	 (q)_\infty U(1;q) &=\sum_{n\geq1} (-1)^n q^{\frac{n(n+1)}{2}} \lp \sum_{m=0}^{n-1} (2m+1)q^{nm}+2n \sum_{m \geq n}q^{nm} \rp
	+\sum_{n\geq1} \sum_{m\geq0} (-1)^{n+1} 2n q^{\frac{n(n+1)}{2}+nm}
	\\
	&=\sum_{n\geq1} (-1)^n q^{\frac{n(n+1)}{2}} \sum_{m=0}^{n-1} (2m+1-2n) q^{nm}
	.
	\end{align*}
	The claim now follows by changing $n \mapsto n+m+1$ and using Proposition \ref{prop:4.1}.
\end{proof}



\begin{remark}
Quantum modular and mock modular properties of $U(1;q)$ are well-understood and therefore follow for $H_{D_4}(q)$.
In particular, this gives (as $t \to 0^+$)
$$
\lp e^{-t} \rp_\infty^3 e^{-\frac{23t}{24}} H_{D_4}\lp e^{-t} \rp =  \sum_{n \geq 0} \frac{T_n}{n!} \left(\frac{-t}{24}\right)^n,
$$
where $T_n$ are Glaisher's numbers \cite[Theorem 1]{BOPR}. 
Moreover, by \cite{Rh}, $(q)_\infty^3 q^{\frac{1}{24}} H_{D_4}(q)$ is a mixed mock modular form.
\end{remark}

\begin{remark}[$\ell$-star graphs]
It is worth noting that for every $\ell$-star graph $X_\ell$, $\ell \geq 3$, we can write 
\begin{equation} \label{l-star}
H_{X_\ell}(q)=\frac{1}{(q)_\infty^\ell} \sum_{n \geq 0} q^n (q)_n^{\ell-1}.
\end{equation}
For $\ell=2$ we obtain the $A_3$-graph function discussed  in Section 3,
via (\ref{tail}). However for $\ell > 3$, we are not aware of any Appell--Lerch type series representation for the sum in (\ref{l-star}).  It would be interesting to determine their quantum modular properties.
\end{remark}

\section{$D_5$ graph series and the proof of \eqref{HD5.2}}

In this section we consider the graph series of type $D_5$:
\begin{equation}\label{HD5}
	H_{D_5}(q)=\sum_{\bm n \in\N_0^5} \frac{q^{n_1 n_2 + n_1 n_3+n_1 n_4+n_4n_5+ n_1+n_2+n_3+n_4+n_5}}{(q)_{n_1} (q)_{n_2} (q)_{n_3} (q)_{n_4}(q)_{n_5}}.
\end{equation}
Our first result is the following Lerch-type sum representation.

\begin{proposition}\label{prop:5.1} We have 
\begin{align*}
H_{D_5}(q)
&=
	\frac{q^{-1}}{(q)_\infty^3}
	\sum_{n\geq 1}
	\frac{(-1)^{n+1}  \lp 1+q^n\rp q^{\frac{n(n+1)}{2}} \left(1-q^{n^2}\right) }{\left(1-q^n\right)^2}.
\end{align*}

\end{proposition}

\begin{proof}
To begin we use \eqref{HD5}, and apply \eqref{109} three times to write 
\begin{equation}\label{rewriteHD5}
H_{D_5}(q) = \frac{1}{(q)_\infty^3} \sum_{n_1,n_4\ge0} q^{n_1n_4+n_1+n_4} (q)_{n_1}
= \frac{1}{(q)_\infty^3} \sum_{n\ge0} \frac{q^{n} (q)_{n}}{1-q^{n+1}}.
\end{equation}
Using \cite[Theorem 8]{Lovejoy1}, with $a=q^2$, $d=b^2\rightarrow0$, and $c=q$ and a direct calculation yields that the following are a Bailey pair relative to $(q^2,q)$, 
\[
\alpha_n := \frac{(-1)^n q^{\frac{n(n+1)}{2}} \left(1+q^{n+1}\right) \left(1-q^{n^2+2n+1}\right)}{1-q^2},\qquad
\beta_n := \frac{1}{\left(q^2\right)_n}.
\]
Inserting this Bailey pair into Lemma \ref{L:BaileysLemma}, 
with $\varrho_1=\varrho_2=q$, yields
\begin{align*}
\sum_{n\ge0}
\frac{q^n (q)_{n}}{1-q^{n+1}}
&=
	\frac{1}{1-q}
	\sum_{n\ge0}
	q^n (q)_{n}^2  \beta_n
=
	\frac{\left(q^2\right)_\infty^2}{(1-q)\left(q^3,q\right)_\infty}
	\sum_{n\ge0}
	\frac{q^n (q)_n^2  \alpha_n }{\left(q^2\right)_n^2}
\\
&=
	\sum_{n\ge0}
	\frac{(-1)^n\left(1+q^{n+1}\right) q^{\frac{n^2+3n}{2}} \left(1-q^{n^2+2n+1}\right)}
	{\left(1-q^{n+1}\right)^2}
	\\
&=
	-q^{-1}
	\sum_{n\ge1}
	\frac{(-1)^n\left(1+q^n\right) q^{\frac{n(n+1)}{2}} \left(1-q^{n^2}\right)}
	{\left(1-q^{n}\right)^2}
.
\end{align*}
Plugging into \eqref{rewriteHD5} gives the claim.
\end{proof}

The next result gives an indefinite theta function representation for the series of interest, proving \eqref{HD5.2}.

\begin{theorem}\label{theorem:5.1}
	We have
	\begin{multline*}
	\sum_{n\geq1} \frac{(-1)^n  \left(1+q^n\right) q^{\frac{n(n+1)}{2}} \left(1-q^{n^2}\right)}{\left(1-q^n\right)^2}\\
	= -\frac{q}{(q)^2_\infty} \lp \sum_{n,m \geq 0} - \sum_{n,m < 0} \rp (-1)^n (n+1)^2 q^{\frac{n^2+3n}{2} +3nm+3m^2+4m}.
\end{multline*}
\end{theorem}
\begin{proof}
The idea is to view this as an identity between modular forms in trivial spaces. For this, denote the left-hand side by $\mathcal{L}(\t)$ and write, using Lemma \ref{AL_E2},
\begin{align*}
	\sum_{n\geq 1} \frac{(-1)^n  \lp 1+q^n \rp q^{\frac{n(n+1)}{2}}}{\lp 1-q^n \rp^2}
	&= \frac{E_2(\t)-1}{24}
	, \\
	-\sum_{n\geq 1} \frac{(-1)^n  \lp 1+q^n \rp q^{\frac{n(3n+1)}{2}}}{\lp 1-q^n \rp^2}
	&=\sum_{n\in\Z \setminus \{0\}} \frac{(-1)^{n+1} q^{\frac{n(3n+1)}{2}}}{\lp 1-q^n\rp^2}
	=\calF(\tau). 
\end{align*}
Thus 
\begin{align*}
	\mathcal{L}(\tau)=\calF(\tau) + \frac{E_2(\t)-1}{24}= \lp \mathcal{F}(\tau)- \frac{1}{24}+\frac{E_2(\tau)}{8} \rp - \frac{E_2(\tau)}{12} .
\end{align*}
Define
\begin{align*}
	\widehat{\mathcal{L}}(\tau):= \wh{\calF}(\t) - \frac{E_2(\t)}{12}.
\end{align*}

We next look at the right-hand side of Theorem \ref{theorem:5.1}, which we denote by $\mathcal{R}(\t)$, and write
\begin{align*}
	\frac 12 \sum_{n,m\in \Z} &\lp \sgn\lp n+ \frac 12 \rp + \sgn \lp m+ \frac 12 \rp \rp (-1)^n (n+1)^2 q^{\frac{n^2+3n}{2} +3nm+3m^2+4m}
	\\
	&=  - \frac{q^{-1}}{2}  \sum_{\bm{n}\in\Z^2} \left(\sgn\left(n_1-\frac 12\right) + \sgn\left(n_2+\frac 12\right)\right) (-1)^{n_1} n_1^2 q^{\frac{n_1^2}{2} + \frac{n_1}{2} + 3n_1n_2 + 3n_2^2 + n_2}
	\\
	& =  \frac{q^{-1}}{8\pi^2} \left[ \frac{\partial^2}{\partial z^2}
	\sum_{\bm{n}\in\Z^2}
	\left(\sgn\left(n_1-\frac 12\right) + \sgn\left(n_2+\frac 12\right)\right) e^{2\pi i\left(n_1\left(\frac 12 +z+\frac \tau2 \right)+n_2\tau\right)} 
%
%
	q^{Q(\bm{n})}
	\right]_{z=0},
\end{align*}
where
$
Q(\bm{n}) := \frac{n_1^2}{2}+3n_1n_2+3n_2^2.
$ Setting $\bm{z} := (-2z,\frac \tau6 +\frac 16 +z)$, $\bm{c_1} := (-2,1)$, $\bm{c_2} := (-1,\frac 13)$, and choosing $y>0$ sufficiently small, we may write
\begin{multline*}
	\sum_{\bm{n}\in\Z^2} \left(\sgn\left(n_1-\frac 12\right) + \sgn\left(n_2+\frac 12\right)\right) e^{2\pi i\left(n_1\left(\frac 12+z+\frac \tau2\right)+n_2\tau\right)} q^{Q(\bm{n})}\\
	= \sum_{\bm{n}\in\Z^2} \left(\sgn\left(B\left(\bm{n}+\frac{\text{Im}(\bm{z})}{v},\bm{c_1}\right)\right) - \sgn\left(B\left(\bm{n}+\frac{\text{Im}(\bm{z})}{v},\bm{c_2}\right)\right)\right) e^{2\pi iB(\bm{n},\bm{z})} q^{Q(\bm{n})}.
\end{multline*}
Define
\begin{align*}
	\Phi(z;\tau) :&= \Theta^{\bm{c_1},\bm{c_2}}_{\left(\begin{smallmatrix} 1&3\\3&6\end{smallmatrix}\right)} (\bm z; \tau), \quad 
	\mathcal{G} (\tau) := \frac{q^{\frac{1}{12}}}{8\pi^2} \left[\frac{\partial^2}{\partial z^2} \Phi (z;\tau) \right]_{z=0},\quad 
	\wh{\mathcal{R}}(\t) := -\frac{\calG(\t)}{\eta(\t)^2}.
\end{align*}
We first show that
\begin{equation}\label{mRm}
	\calL^-(\t) := \wh{\calL}(\t) - \calL(\t) = \wh{\calR}(\t) - \calR(\t) =: \calR^-(\t).
\end{equation}
For this, we rewrite the right-hand side of this identity. We start by computing
\begin{multline}\label{diffP}
	\left[ \frac{\partial^2}{\partial z^2} \Phi (z;\tau)\right]_{z=0}
	=\sum_{\bm n \in \Z^2} (-1)^{n_1}   q^{\frac{n_1^2}{2} + \frac{n_1}{2} + 3n_1n_2 + 3n_2^2 +n_2}
	\\ 
	\times \left[ \frac{\partial^2}{\partial z^2} \lp \lp E \lp \lp n_1-\frac{2y}{v} \rp \sqrt{v} \rp
	+E \lp \lp n_2 + \frac 16 + \frac yv \rp \sqrt{6v} \rp \rp  e^{2\pi i n_1z} \rp \right]_{z=0}.
\end{multline}
For the first summand, write
\begin{align*}
	\frac{n_1^2}{2} + \frac{n_1}{2} + 3n_1n_2 + 3n_2^2 +n_2=3\lp n_2+\frac{n_1}{2} \rp^2 + \lp n_2+ \frac{n_1}{2} \rp - \frac{n_1^2}{4},
\end{align*}
make the change of variables $n_1\mapsto2n_1+\delta$, $\delta \in \{0,1\}$, $n_1\in\Z$, and then let $n_2 \mapsto n_2-n_1$. Using the second identity in \eqref{identity1.1}, the contribution to $\calR^-$ is 
\begin{multline*}
	-\frac{1}{\sqrt{\pi}} \sum_{\bm n \in \Z^2} (-1)^{n_1} f(n_1) q^{3\lp n_2+\frac{n_1}{2} \rp^2 + \lp n_2+\frac{n_1}{2} \rp - \frac{n_1^2}{4}}\\
	=-\frac{1}{\sqrt{\pi}} \sum_{\delta \in \{0,1\}} (-1)^\delta \sum_{n_1\in\Z} f(2n_1+\delta) q^{-\frac 14 (2n_1+\delta)^2} \sum_{n_2 \in\Z} q^{3\left(n_2+\frac \delta2\right)^2+n_2+\frac \delta2},
\end{multline*}
where
\begin{align*}
	f(n):= \left[ \frac{\partial^2}{\del z^2} \lp \sgn \lp n-\frac{2y}{v} \rp \Gamma \lp \frac 12 , \pi \lp n - \frac{2y}{v} \rp^2 \rp e^{2\pi inz} \rp \right]_{z=0}.
\end{align*}
By changing $n_1\mapsto-n_1-\delta$ one sees that the sum on $n_1$ vanishes (since $f$ is an odd function). 

For the second term in \eqref{diffP}, write 
\begin{align*}
	\frac{n_1^2}{2} + \frac{n_1}{2} + 3n_1n_2 + 3n_2^2 +n_2
	=\frac 12 \lp n_1+3n_2\rp^2 + \frac 12 \lp n_1+ 3n_2 \rp - \frac{3n_2^2}{2}- \frac{n_2}{2}.
\end{align*}
Then we shift $n_1 \mapsto n_1-3n_2$ to obtain that the contribution of the second term to $\calR^-$ is
\begin{multline}\label{second}
	-\frac{1}{\sqrt{\pi}} \sum_{\bm n \in \Z^2} (-1)^{n_1} g(n_1,n_2) q^{\frac 12 (n_1+3n_2)^2+ \frac 12 (n_1+3n_2) - \frac{3n_2^2}{2}-\frac{n_2}{2}}
	\\
	=-\frac{1}{\sqrt{\pi}} \sum_{\bm n \in\Z^2} (-1)^{n_1+n_2} g(n_1-3n_2,n_2) q^{\frac{n_1^2}{2}+ \frac{n_1}{2} - \frac{3n_2^2}{2} - \frac{n_2}{2}},
\end{multline}
where
\begin{align*}
	g(\bm{n}):= \left[ \frac{\del^2}{\del z^2} \lp \sgn \lp n_2+\frac 16+ \frac yv \rp \Gamma \lp \frac 12 , 6\pi \lp n_2+\frac 16 + \frac yv \rp^2 v \rp e^{2\pi i n_1 z} \rp  \right]_{z=0}.
\end{align*}
Making the change of variables $n_1\mapsto-n_1-1$, we see that
\begin{align*}
	\sum_{n_1 \in\Z} (-1)^{n_1} h\lp n_1+\frac 12 \rp q^{\frac{n_1(n_1+1)}{2}}=0
\end{align*}
for any even function $h$.
Thus we obtain that \eqref{second} equals
\begin{multline*}
	- 4\sqrt{\pi}i \sum_{n_1\in\Z} (-1)^{n_1} \lp n_1 + \frac 12 \rp q^{\frac{n_1(n_1+1)}{2}} 
	\\ \times
	\sum_{n_2\in\Z} (-1)^{n_2} \sgn \lp n_2+ \frac 16 \rp q^{-\frac{3n_2^2}{2}-\frac{n_2}{2}} 
	\left[ \frac{\del}{\del z} \lp  \Gamma \lp \frac 12 , 6\pi \lp n_2+\frac 16 + \frac yv \rp^2 v \rp e^{-6\pi i \lp n_2 +\frac 16 \rp z} \rp  \right]_{z=0}.
\end{multline*}
The sum on $n_1$ is
\begin{align*}
	q^{-\frac{1}{8}} \sum_{n_1 \geq 0} (-1)^{n_1} (2n_1+1) q^{\frac{1}{8}(2n_1+1)^2} = q^{-\frac{1}{8}} \eta(\t)^3.
\end{align*}
The sum on $n_2$ is
\begin{equation}\label{sumn2}
	q^{\frac{1}{24}} \sum_{n_2 \in \Z+\frac 16} (-1)^{n_2-\frac 16} \sgn (n_2) q^{-\frac{3n_2^2}{2}}
	\left[ \frac{\del}{\del z} \lp  \Gamma \lp \frac 12 , 6\pi \lp n_2+\frac yv \rp^2 v \rp e^{-6\pi i n_2 z} \rp  \right]_{z=0}.
\end{equation}
We then compute, using \eqref{identity1.2},
\begin{align*}
	&\left[\frac{\del}{\del z} \lp  \Gamma \lp \frac 12 , 6\pi \lp n_2+\frac yv \rp^2 v \rp e^{-6\pi i n_2 z} \rp  \right]_{z=0}
	= 3\pi in_2 \Gamma\left(-\frac 12, 6\pi n_2^2 v\right).
\end{align*}
Thus \eqref{sumn2} equals (upon changing $n_2 \mapsto -n_2$)
\begin{multline*}
	3\pi i q^{\frac{1}{24}} \sum_{n_2\in\Z+\frac 16} (-1)^{n_2-\frac 16} |n_2| \Gamma\left(-\frac 12, 6\pi n_2^2 v\right) q^{-\frac{3n_2^2}{2}}\\
	= 3\pi i q^{\frac{1}{24}} \sum_{n_2\in\Z-\frac 16} (-1)^{n_2+\frac 16} |n_2| \Gamma\left(-\frac 12, 6\pi n_2^2 v\right) q^{-\frac{3n_2^2}{2}}.
\end{multline*}
From this we obtain \eqref{mRm}.

We next determine the transformation laws of $\widehat{\mathcal{L}}$ and of $\widehat{\mathcal{R}}$. By Proposition \ref{Duke} and \eqref{E2},
\begin{equation}\label{R}
	\widehat{\mathcal{L}} (\tau+1) = \widehat{\mathcal{L}}(\tau), \qquad
	\widehat{\mathcal{L}} \lp -\frac{1}{\tau} \rp =\tau^2 \widehat{\mathcal{L}}(\tau) + \frac{i\tau}{2\pi}.
\end{equation}
We show that $\wh{\mathcal{R}}$ satisfies the same transformation laws as $\wh{\mathcal{L}}$.
By Proposition \ref{prop:elliptic} we obtain
\begin{align*}
	\Phi \lp \frac z\tau ;-\frac 1\tau \rp 
	& = \frac{\tau}{\sqrt{3}}  \sum_{\ell\pmod{3}} e^{\frac{2\pi i}{\tau} Q\lp -2z, -\frac 16+ \frac \tau6 +z+ \frac{\ell\tau}{3} \rp} \Theta_{\left(\begin{smallmatrix} 1&3\\3&6\end{smallmatrix}\right)} \lp  \lp -2z, -\frac 16 + \frac \tau6 +z +\frac{\ell\tau}{3} \rp; \tau \rp.
\end{align*}
Changing $z \mapsto -z$ and using Proposition \ref{prop:elliptic} (2), we compute
\begin{multline*}
\left[\frac{\partial^2}{\partial z^2} \left(e^{\frac{2\pi i}{\t} Q\left(-2z, -\frac 16+\frac \tau6+z+ \frac{\ell\tau}{3}\right)} 
\Theta_{\left(\begin{smallmatrix} 1&3\\3&6\end{smallmatrix}\right)} \lp  \lp -2z, -\frac 16 + \frac \tau6 +z +\frac{\ell\tau}{3} \rp; \tau \rp
\right)\right]_{z=0}
	\\ =-\left[\frac{\partial^2}{\partial z^2} \left(e^{\frac{2\pi i}{\t} Q\left(-2z, \frac 16-\frac \tau6+z- \frac{\ell\tau}{3}\right)} 
	 \Theta_{\left(\begin{smallmatrix}1&3\\3&6\end{smallmatrix}\right)}\left(\left(-2z,\frac 16-\frac \tau6+z- \frac{\ell\tau}{3} \right);\tau\right)\right)\right]_{z=0}.
\end{multline*}
Thus
\begin{multline*}
	\mathcal{G} \left(-\frac{1}{\tau}\right) 
	\\ 
	= -\tau^3\frac{e^{-\frac{\pi i}{6\tau}}}{8\sqrt{3}\pi^2} 
	\sum_{\ell \pmod 3}
	\left[\frac{\partial^2}{\partial z^2} \left(e^{\frac{2\pi i}{\t} Q\left(-2z, \frac 16-\frac \tau6+z- \frac{\ell\tau}{3}\right)} \Theta_{\left(\begin{smallmatrix}1&3\\3&6\end{smallmatrix}\right)}\left(\left(-2z,\frac 16-\frac \tau6+z- \frac{\ell\tau}{3} \right);\tau\right)\right)\right]_{z=0}.
\end{multline*}
Now we choose $\ell\in\{0,-1,-2\}$.
Using Proposition \ref{prop:elliptic} (2) and (3), we see that the contribution for $\ell=-2$
is an odd function evaluated at zero, and as such vanishes.

Using Proposition \ref{prop:elliptic} again to relate the remaining theta functions we obtain
\begin{equation}\label{G}
	\calG\left(-\frac 1\t\right) = -\frac{i\t^3}{8\pi^2} q^{\frac{1}{12}}
\left[\frac{\partial^2}{\partial z^2}  \left(e^{-\frac{2\pi iz^2}{\t}}  \Theta_{\left(\begin{smallmatrix} 1&3\\3&6\end{smallmatrix}\right)} \left(\left(2z,\frac 16 + \frac \tau6 +z\right);\tau\right)\right)\right]_{z=0}.
\end{equation}
Now
\begin{multline*}
	\left[\frac{\partial^2}{\partial z^2}  \left(e^{-\frac{2\pi iz^2}{\t}}  \Theta_{\left(\begin{smallmatrix} 1&3\\3&6\end{smallmatrix}\right)} \left(\left(2z,\frac 16 + \frac \tau6 +z\right);\tau\right)\right)\right]_{z=0}\\
	= \left[\frac{\partial^2}{\partial z^2}  \Theta_{\left(\begin{smallmatrix} 1&3\\3&6\end{smallmatrix}\right)} \left(\left(2z,\frac 16 + \frac \tau6 +z\right);\tau\right)\right]_{z=0} - \frac{4\pi i}{\t} \Theta_{\left(\begin{smallmatrix} 1&3\\3&6\end{smallmatrix}\right)} \left(\left(0,\frac 16 + \frac \tau6\right);\tau\right).
\end{multline*}
The first summand contributes
\begin{equation}\label{againG}
	-i\t^3 \calG(\t).
\end{equation}

We next claim that
\begin{equation}\label{TE}
	\Theta_{\left(\begin{smallmatrix} 1&3\\3&6\end{smallmatrix}\right)}  \lp \lp 0,\frac 16 + \frac \t6 \rp;\tau \rp = (q)^2_\infty.
\end{equation}
For this, we make the same changes of variables as in the proof of \eqref{mRm} to obtain that
\begin{equation*}
	\Theta_{\left(\begin{smallmatrix} 1&3\\3&6\end{smallmatrix}\right)}\lp \lp 0,\frac 16 + \frac \t6 \rp ; \tau \rp\\
	=\sum_{\bm{n} \in \Z^2} (-1)^{n_1} \lp \sgn(n_1) + \sgn \lp n_2 + \frac 12 \rp \rp q^{\frac{n_1^2}{2} + \frac{n_1}{2} + 3n_1n_2 + 3n_2^2 +n_2}.
\end{equation*}
One can now show (as above) that both sides of \eqref{TE} satisfy the same transformation law and lie in a one-dimensional space. Computing one coefficient then gives that they are equal.

Using \eqref{R}, \eqref{G}, and \eqref{againG}, we obtain that
\[
\wh{\calR}(\t+1) = \wh{\calR}(\t),\qquad \wh{\calR}\left(-\frac 1\t\right) = \t^2 \wh{\calR}(\t) + \frac{i\t}{2\pi}.
\]
This shows that $\calR(\t)-\calL(\t)$ is a weakly holomorphic modular form of weight two for $\SL_2(\Z)$. Since one can prove that it does not grow it has to be zero.
\end{proof}

\begin{remark}
	Propositions \ref{prop:5.1} and \ref{theorem:5.1} give that $H_{D_5}$ is a mixed mock modular form, as claimed in Theorem \ref{thm:1.3}.
\end{remark}

\section{Graphs with multiple edges: Kontsevich--Zagier type series}

In this part we contemplate graph series with multiple edges. Series of this type do not connect directly with the geometry of jet schemes, but they do naturally appear in vertex algebras (see Theorem \ref{principal}).
Here we focus on the two simplest examples coming from the graph
\begin{center}
$\bullet = \bullet$
\end{center}
of type $B_2$, and from the graph
\begin{center}
$\bullet - \bullet =\bullet$
\end{center} 
of type $B_3$. In the setup of principal subspaces, we consider additional $q$-series arising from cosets in the dual lattices (these compute characters of modules). Thus for $B_2$ we obtain three $q$-series   \begin{align*}
F_1(q):=\sum_{n_1,n_2 \geq 0} \frac{q^{2n_1n_2+n_1+n_2 }}{(q)_{n_1}(q)_{n_2}}, \quad
F_2(q):=\sum_{n_1,n_2 \geq 0} \frac{q^{2n_1n_2+n_1+2n_2}}{(q)_{n_1}(q)_{n_2}}, \quad
F_3(q):=\sum_{n_1,n_2 \geq 0} \frac{q^{2n_1n_2+2n_1+2n_2}}{(q)_{n_1}(q)_{n_2}}.
\end{align*}

Let us now recall two remarkable  false theta functions used to compute unified WRT invariants of the Poincar\'e $3$-sphere \cite{LZ}. For this we let
\[
\chi_+(n) := 
\begin{cases}
	(-1)^{\left\lfloor\frac{n}{30}\right\rfloor} &\text{if } n^2 \equiv 1 \pmod{120},\\
	0 &\text{otherwise},
\end{cases}\qquad
\chi_-(n) := 
\begin{cases}
	(-1)^{\left\lfloor\frac{n}{30}\right\rfloor} &\text{if } n^2 \equiv 49 \pmod{120},\\
	0 &\text{otherwise}.
\end{cases}
\]
Then the two $q$-series
$$
\widetilde{\Theta}_{+}(q):=\sum_{n \geq 1} \chi_{+}(n)q^{\frac{n^2-1}{120}}, \qquad
\widetilde{\Theta}_{-}(q):=\sum_{n \geq 1} \chi_{-}(n)q^{\frac{n^2-49}{120}},
$$
combine into a vector-valued quantum modular form of weight $\frac12$ \cite[Section 4]{LZ}. The following proposition relates the functions $F_1$, $F_2$, and $F_3$ to these false theta functions.
\begin{proposition}\label{prop:6.1}
We have
$$
F_1(q)=\frac{\widetilde{\Theta}_{-}(q)}{(q)_\infty}
, \quad
F_2(q)=\frac{q^{-1}\left(\widetilde{\Theta}_{+}(q)-1\right)}{(q)_\infty}, \quad F_3(q)=\frac{q^{-2}\left(\widetilde{\Theta}_-(q)- \widetilde{\Theta}_+(q) \right)}{(q)_\infty}.
$$ 
\end{proposition}

\begin{proof} 
Using \eqref{109}, we obtain 
$$
F_2(q)=\sum_{n_1,n_2 \geq 0} \frac{q^{n_1+(2n_1+2)n_2}}{(q)_{n_1}(q)_{n_2}}=\sum_{n_1 \geq 0} \frac{q^{n_1}}{(q)_{n_1} \left(q^{2n_1+2}\right)_\infty} =\frac{1}{(q)_\infty} \sum_{n \geq 0} q^n \left(q^{n+1}\right)_{n+1}.
$$
The sum is known to be $q^{-1}(\widetilde{\Theta}_{+}(q)-1)$ by (3.14) of \cite{Hikami}.  For the first identity, we write
$$
F_1(q)=\sum_{n_1,n_2 \geq 0} \frac{q^{n_1+(2n_1+1)n_2}}{(q)_{n_1}(q)_{n_2}}=\sum_{n_1 \geq 0} \frac{q^{n_1}}{(q)_{n_1} \left(q^{2n_1+1}\right)_\infty}=\frac{1}{(q)_\infty} \sum_{n\ge0} q^{n} \left(q^{n+1}\right)_n,
$$
and use (3.13) of \cite{Hikami}.
Easy manipulations then yield $$F_3(q)=q^{-2} F_1(q)-q^{-1} F_2(q)-\frac{q^{-2}}{(q)_\infty},$$ 
which implies the formula for $F_3(q)$.
\end{proof}

For a graph $\Gamma$ of type $B_3$ we record similar identities.  We first consider a slightly shifted version of $H_{\Gamma}(q)$ given by 
$$
H_1(q):=\sum_{\bm n \in \N_0^3} \frac{q^{2n_1n_2+n_2 n_3+n_1+2n_2+n_3}}{(q)_{n_1}(q)_{n_2}(q)_{n_3}}.
$$
As before, using Euler's identity (\ref{109}) gives
\[
H_1(q) = \frac{q^{-1}}{(q)_\infty} \sum_{n_1 \geq 0} \frac{q^{n_1+1}}{\left(1-q^{2n_1+2}\right)(q)_{n_1}}=\frac{q^{-1}}{(q)_\infty} \sum_{n \geq 0} \frac{q^{n+1}}{\left(1+q^{n+1}\right)(q)_{n+1}}
= \frac{q^{-1}}{(q)_\infty} \sum_{n \geq 1} \frac{q^n}{\left(1+q^n\right)(q)_n}.
\]
We can rewrite the right-hand side using a sum of tails identity \cite[Theorem 4.1]{AF} as
$$
H_1(q)=\frac{q^{-1}}{(q)_\infty^2} \sum_{n \geq 1} (-1)^n \left( (q)_n - (q)_\infty \right).
$$
By \eqref{id2} with $\z=-1$, we have
$$
\sum_{n \geq 0} (-1)^n \left( (q)_n - (q)_\infty \right){=}\sum_{n \geq 1} q^{2n-1}(q)_{2n-2}.
$$
This can be further expressed, using \eqref{tail}, as
$$
\frac12 \left(-(q)_\infty+1+\sum_{n \geq 0} (-1)^n q^{n+1}(q)_n \right).
$$
The last sum, $\sigma(q)=1+\sum_{n \geq 0} (-1)^n q^{n+1}(q)_n$, is a quantum modular of weight zero (see the examples in \cite{ZaQ}). 
For the graph series 
$$
H_2(q):=\sum_{\bm n \in \N_0^3} \frac{q^{2n_1n_2+n_2 n_3+n_1+n_2+n_3}}{(q)_{n_1}(q)_{n_2}(q)_{n_3}},
$$
we first deduce that
$$
H_2(q)=\frac{1}{(q)_\infty} \sum_{n_1 \geq 0} \frac{q^{n_1}}{(1-q^{2n_1+1})(q)_{n_1}}=\frac{1}{(1-q)(q)_\infty}+\frac{1}{(q)_\infty^2} \sum_{n \geq 0} q^n((q)_{2n}-(q)_\infty),
$$
by \cite[Theorem 4.1]{AF}  and  \cite[Theorem 2]{Za}. 
Then using \eqref{id2} (we extract the even powers of $\z$ and let $\z=q$) gives 
\begin{align*}
\frac{1}{(1-q)(q)_\infty}+\frac{1}{(q)_\infty^2} \sum_{n \geq 0} q^n\left((q)_{2n}-(q)_\infty\right) 
=\frac{1}{(1-q)(q)^2_\infty}\left(1-\sum_{n \geq 0} \left( q^{3n+2} (q)_{2n}+q^{3n+3}(q)_{2n+1} \right) \right),
\end{align*}
where we empoy that (by \eqref{tail})
$$
\sum_{n \geq 0} \left(q^{2n+1}(q)_{2n}+q^{2n+2}(q)_{2n+1}\right)=1-(q)_\infty.
$$

\section{Further examples and the proof of \eqref{HE6}}
Here we consider a few more complicated graphs.

\subsection{$3$-cycle}

As shown in \cite{JM}, for the three cycle graph $\Gamma$ we have
\begin{gather*}
\sum_{\bm n \in \N_0^3}
\frac{q^{n_1n_2+n_1n_3+n_2n_3+n_1+n_2+n_3}}{(q)_{n_1}(q)_{n_2}(q)_{n_3}}
= \frac{1}{(q)_\infty} \sum_{n\ge0}\frac{q^n}{(q^{n+1})_{n+1}},
\end{gather*}
where the sum on the right-hand side is $\chi_1(q)$, a fifth order mock theta function of Ramanujan. Such mock theta functions were introduced in Ramanujan's last letter to Hardy.

Since ${\rm det}(C_\Gamma)=2$ for $C_\Gamma=\left( \begin{smallmatrix} 0 & 1 & 1 \\ 1 & 0 & 1 \\ 1 & 1 & 0 \end{smallmatrix} \right)$, it is natural to consider an additional $q$-series coming from 
the nontrivial coset.
We obtain a related identity for another fifth order mock theta function $\chi_0$
\begin{gather*}
\sum_{\bm n \in \N_0^3}
\frac{q^{n_1n_2+n_1n_3+n_2n_3+2n_1+n_2+n_3}}{(q)_{n_1}(q)_{n_2}(q)_{n_3}}
= \frac{1}{(q)_\infty} \sum_{n\ge0}\frac{q^n}{(q^{n+2})_{n+1}}=\frac{q^{-1}}{(q)_\infty} \left(\chi_0(q)-1) \right).
\end{gather*}
Interestingly, $\chi_0(q)$ and $\chi_1(q)$ combine into a vector-valued quantum modular form \cite{Hikami}.

\subsection{Graph series of $E_6$}
Now $\Gamma$ is 
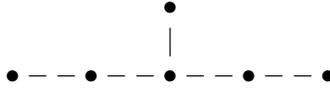
\begin{center}
$        \bullet        $

$        |        $

$\bullet--\bullet--\bullet--\bullet--\bullet$
\end{center}
and the graph series is 
$$H_{E_6}(q)=\sum_{\bm n \in \N_0^6}  \frac{q^{n_1 n_2+n_1 n_3+n_1n_4+n_2 n_5+n_3 n_6+n_1+n_2+n_3+n_4+n_5+n_6}}{(q)_{n_1}(q)_{n_2}(q)_{n_3}(q)_{n_4}(q)_{n_5}(q)_{n_6}} .$$

We next prove \eqref{HE6}.
\begin{proof}[Proof of \eqref{HE6}]
We compute
\begin{align*}
H_{E_6}(q)
&=
	\sum_{\bm n \in \N_0^6}
	\frac{ q^{n_1(n_2+n_3+n_4)+n_2n_5+n_3n_6+n_1+n_2+n_3+n_4+n_5+n_6}  }
	{(q)_{n_1}(q)_{n_2}(q)_{n_3}(q)_{n_4}(q)_{n_5}(q)_{n_6}}
\\
&=
	\frac{1}{(q)_\infty^2}
	\sum_{n_1,n_2,n_3,n_4\ge0}
	\frac{ q^{n_1(n_2+n_3+n_4)+n_1+n_2+n_3+n_4}  }
	{(q)_{n_1}(q)_{n_4}}
=
	\frac{1}{(q)_\infty^3}
	\sum_{n_1,n_2,n_3\ge0}
	q^{n_1(n_2+n_3)+n_1+n_2+n_3} 
\\
&=
	\frac{1}{(q)_\infty^3}
	\sum_{n\ge0}
	\frac{q^{n} }{\left(1-q^{n+1}\right)^2}
=
	\frac{q^{-1}}{(q)_\infty^3}
	\sum_{n\ge1}
	\frac{q^n}{\left(1-q^n\right)^2}
=
	\frac{q^{-1}}{(q)_\infty^3}
	\sum_{n\ge1}
	\frac{nq^n}{1-q^{n}}
. \qedhere
\end{align*}
\end{proof}

We finish with two examples coming from affine Dynkin diagrams.

\subsection{$\tt{H}$-graph (or  $D_5^{(1)}$)}
Here we consider $\Gamma$ to be an $H$ graph as in the picture:
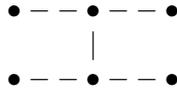
\begin{center}
$        \bullet   -- \bullet -- \bullet    $

$       |   $

$        \bullet   -- \bullet -- \bullet    $
\end{center}
This graph series is given by 
$$
H_\Gamma(q)=\sum_{\bm n \in \N_0^6} \frac{ q^{n_1n_2+n_1n_3+n_1 n_4 + n_4n_5+n_4n_6+n_1+n_2+n_3+n_4+n_5+n_6} }{(q)_{n_1}(q)_{n_2}(q)_{n_3}(q)_{n_4}(q)_{n_5}(q)_{n_6}}.
$$
From this we easily find that
$$
H_\Gamma(q)=\frac{1}{(q)_\infty^4} \sum_{n,m \geq 0 }   q^{mn+m+n} (q)_m (q)_n.
$$
It would be interesting to explore modular properties of this double series. We believe that it is a (higher depth) quantum modular form.

\subsection{$T_2$-graph (or $E_6^{(1)}$)} The next example is obtained by adding an extra node to an $E_6$ graph.
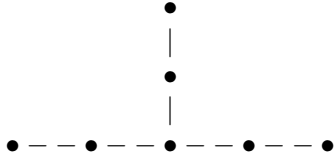
\begin{center}
$        \bullet        $

$        |        $

$        \bullet        $

$        |        $

$\bullet--\bullet--\bullet--\bullet--\bullet$
\end{center}
Here we get 
\begin{align*}
 H_{T_2}(q)&=
\sum_{\bm n \in \N_0^7} \frac{q^{n_1 n_6+n_2 n_4+n_2 n_5+n_2 n_6+n_3 n_4+n_3 n_5+n_1+n_2+n_3+n_4+n_5+n_6+n_7}}{(q)_{n_1}(q)_{n_2}(q)_{n_3}(q)_{n_4}(q)_{n_5}(q)_{n_6}(q)_{n_7}} \\
&=
\frac{1}{(q)^4_\infty}\sum_{n_1,n_2,n_3 \geq 0} q^{n_1+n_2+n_3} (q)_{n_1+n_2+n_3}= \frac{1}{2 (q)^4_\infty}\sum_{n \geq 0} \left(n^2+3n+2\right) q^{n}(q)_n.
\end{align*}
The last sum can be expressed as a sum of tails by differentiating (\ref{id2}) with respect to $\zeta$ and then letting $\zeta=1$. This immediately gives
\begin{equation*}
H_{T_2}(q)= \frac{q^{-1}}{(q)^4_\infty} \sum_{n \geq 0} (n+1) \left( (q)_n - (q)_\infty \right).
\end{equation*}
 Since $\sum_{n \geq 0} ( (q)_n-(q)_\infty)$ is a quantum modular form, it would be interesting to investigate modular properties of the sum $$\sum_{n \geq 0} n \left((q)_n-(q)_\infty \right).$$

\section{Conclusion and open questions}

We hope that this paper generates interest in graph series and their modular properties. 
However, unlike Nahm sums, they seem not to give rise to usual modular forms even for very simple  graphs. 
Instead we obtain interesting combinations of mixed quantum and mock modular forms. This raises a natural question:

\smallskip

{\em Is there a simple graph $\Gamma$, not totally disconnected, and $a \in \mathbb{Q}$, such that $q^a H_{\Gamma}(q)$ is a modular form?}

\smallskip

We point out that for many examples we are not aware of any modular properties. This is the case for the following 
graphs:
\begin{align*}
& \frac{q^{-1}}{(q)^4_\infty} \sum_{n \geq 0} (n+1) \left( (q)_n - (q)_\infty \right), \quad  \frac{1}{(q)_\infty^\ell} \sum_{n \geq 0} q^n (q)_n^{\ell-1}, \quad \ell \geq 4 \\
& \frac{1}{(q)_\infty^4} \sum_{n,m \geq 0 }   q^{nm+m+n} (q)_n (q)_m, \quad \frac{1}{(q)_\infty} \sum_{n,m \geq 1} \frac{q^{nm}}{(q)_{n+m-1}},
\end{align*}
corresponding to $T_2$, $\ell$-star graphs $X_{\ell}, \ell \geq 4$, $H$, and $4$-cycle graphs, respectively. We hope to return to these examples in future work.


\begin{thebibliography}{99}


	

	\bibitem{AndrewsBook}	
	G. Andrews,
	\newblock \emph{{$q$}-series: their development and application in analysis,
  		number theory, combinatorics, physics, and computer algebra}, vol.~66 of
  	\emph{CBMS Regional Conference Series in Mathematics}.
	\newblock Published for the Conference Board of the Mathematical Sciences,
  	Washington, DC; by the American Mathematical Society, Providence, RI (1986).

        \bibitem{stacks} G. Andrews, {\em Stacked lattice boxes}, Annals of Combinatorics \textbf{3} (1999), 115--130.

	\bibitem{Andrews1} G. Andrews, {\it Concave and convex compositions}, The Ramanujan Journal {\bf 31} (2013), 67--82.
	
	\bibitem{AGL} G. Andrews, F. Garvan, and J. Liang, {\em Self-conjugate vector partitions and the parity of the spt-function}, Acta Arithmetica {\bf 158} (2013), 199--218.
	
	\bibitem{AF} G. Andrews and P. Freitas, {\em Extension of Abel’s Lemma with q-series implications}, The Ramanujan Journal \textbf{10} (2005), 137–152.
	
	\bibitem{ARE} G. Andrews, J. van Ekeren, and R. Heluani, {\em The singular support of the Ising model}, arXiv:2005.10769.
	
	\bibitem{Arakawa}  T. Arakawa, K. Kawasetsu, and J. Sebag,  {\it A question of Joseph Ritt from the point of view of vertex algebras}, arXiv:2009.04615.

  	\bibitem{Gorsky} Y. Bai, E. Gorsky, and O. Kivinen, {\em Quadratic ideals and Rogers–Ramanujan recursions}, The Ramanujan Journal {\bf 52} (2020), 67--89.

	
	
	\bibitem{BO} P. Beirne and R. Osburn, {\em q-series and tails of colored Jones polynomials}, Indagationes Mathematicae {\bf 28} (2017), 247--260.
	
	\bibitem{BC} S. Bettin and B. Conrey,  {\em Period functions and cotangent sums}, Algebra Number Theory \textbf{7} (2013), 215--242.
	
	\bibitem{35} K. Bringmann, {\it Taylor coefficients of non-holomorphic Jacobi forms and applications}. Research in the Mathematical Sciences {\bf 5:15} (2018). 
	
	\bibitem{M} C. Bruschek, H. Mourtada,  and J. Schepers, {\em Arc spaces and the Rogers–Ramanujan identities}, The Ramanujan Journal \textbf{30} (2013), 9--38.
	
	\bibitem{BOPR} J. Bryson, K. Ono, S. Pitman, and R. Rhoades, {\it Unimodal sequences and quantum and mock modular forms}, Proceedings of the National Academy of Sciences {\bf 109} (2012), 16063--16067.

  	\bibitem{CGZ}  F. Calegari, S. Garoufalidis, and D. Zagier, {\em Bloch groups, algebraic K-theory, units, and Nahm's Conjecture}, arXiv:1712.04887.
  	
  	\bibitem{CLM} C. Calinescu, J. Lepowsky, and A. Milas, {\em Vertex-algebraic structure of the principal subspaces of certain-modules, I: level one case}, International Journal of Mathematics {\bf 19} (2008), 71--92.
       
       \bibitem{CLM2} C. Calinescu, J. Lepowsky, and A. Milas, {\em Vertex-algebraic structure of the principal subspaces of certain $A_1^{(1)}$-modules}, II: Higher-level case. Journal of Pure and Applied Algebra, {\bf 212} (8), (2008), 1928--1950.
       
\bibitem{CNV}  S. Cecotti, A. Neitzke, and C. Vafa, {\em R-twisting and 4d/2d correspondences}, arXiv:1006.3435.
       
   	\bibitem{CF} I. Cherednik and B. Feigin, {\em Rogers–Ramanujan type identities and Nil-DAHA}, Advances in Mathematics {\bf 248} (2013), 1050--1088.
      
      \bibitem{FF} B. Feigin and E. Frenkel, {\em Coinvariants of nilpotent subalgebras of the Virasoro algebra and partition identities}, Adv. Sov. Math 16, no. 1003 (1993): 139-148.
          
  	\bibitem{FS} B. Feigin and A. Stoyanovsky, {\em Quasi-particles models for the representations of Lie algebras and geometry of flag manifold}, arXiv hep-th/9308079.
       
   \bibitem{Fine} N. Fine, {\em Basic Hypergeometric series}, American Mathematical Society, 1988.
       
	\bibitem{GL} S. Garoufalidis and T. Le, {\it Nahm sums, stability and the colored Jones polynomial}, Research in the Mathematical Sciences {\bf 2} (2015), 1--55.

 	\bibitem{Gupta} R. Gupta, {\em On sum-of-tails identities}, arXiv:2002.00447.

      
	\bibitem{Hikami} K. Hikami, {\em Mock (false) theta functions as quantum invariants}, Regular and Chaotic Dynamics \textbf{10} (2005), 509--530.

	\bibitem{JM} C. Jennings-Shaffer and A. Milas, {\it  Further q-series identities and conjectures relating false theta functions and characters}, {Contemporary Mathematics}, to appear; arXiv:2005.13620.
	
	\bibitem{Kawasetsu} K. Kawasetsu, {\em The free generalized vertex algebras and generalized principal subspaces}, Journal of Algebra \textbf{ 444} (2015), 20--51.
	
	\bibitem{KO} A. Keilthy and R. Osburn, {\em Rogers–Ramanujan type identities for alternating knots},  Journal of Number Theory {\bf 161} (2016), 255--280.
	
	
	\bibitem{KS}  M. Kontsevich and Y. Soibelman, {\em Stability structures, motivic Donaldson-Thomas invariants and cluster transformations}, arXiv:0811.2435.
	
	\bibitem{Kucharski} P. Kucharski, P. Reineke, M. Sto\v{s}i\'c, and P. Sulkowski, {\em Knots-quivers correspondence,} Advances in Theoretical and Mathematical Physics {\bf 23} (2020), 1849--1902.
	
	\bibitem{LZ} R. Lawrence and D. Zagier, {\it Modular forms and quantum invariants of 3-manifolds}, Asian Journal of Mathematics {\bf 3} (1999), 93--108.
	
	\bibitem{Lerch} M. Lerch,  {\em Pozn\'amky k theorii funkc\'{i} elliptick\'ych}, Rozpravy \v{C}eske Akademie II. 1 {\bf 24} (1892), 465–480.
	
	\bibitem{Li-Thesis} H. Li,  {\it Associated Varieties of Vertex Superalgebras and Equivariant Oriented Cohomology}, PhD thesis, SUNY-Albany, to appear. 

	\bibitem{Li} H. Li, {\it  Some remarks on associated varieties of vertex operator superalgebras,} European Journal of Mathematics, to appear; arXiv:2007.04522.

	\bibitem{LM1} H. Li and A. Milas, {\it  Jet schemes, quantum dilogarithm and Feigin-Stoyanovsky's principal subspaces,}  arXiv:2010.02143.
	
	\bibitem{LM2} H. Li and A.Milas, {\it Principal subspaces of some non-standard modules and their arc spaces}, preprint.
	
	\bibitem{Lovejoy1}J.~Lovejoy, {\it Lacunary partition functions}, Mathematical Research Letters \textbf{9} (2002), 2--3.
	
	
	\bibitem{MP} A. Milas and M. Penn, {\it Lattice vertex algebras and combinatorial bases: general case and $W$-algebras,} The New York Journal of Mathematics {\bf18} (2012), 621--650.

    \bibitem{NRT} W. Nahm, A. Recknagel, and M. Terhoeven, {\em  Dilogarithm identities in conformal field theory}, Modern Physics Letters A {\bf 8} (1993), 1835--1847.
    
	\bibitem{NR} H. Ngo and R. Rhoades, \textit{Integer partitions, probabilities and quantum modular forms}, Research in the Mathematics Sciences \textbf{4:17} (2017).

	\bibitem{Penn} M. Penn, {\em Lattice vertex superalgebras, I: Presentation of the principal subalgebra}, Communications in Algebra {\bf 42} (2014), 933--961.
	
	
	\bibitem{Rh} R. Rhoades, \textit{Asymptotics for the number of strongly unimodal sequences}, International Mathematics Research Notices (2014), 700--719.
	
	\bibitem{WZ}  O. Warnaar and W. Zudilin, {\em Dedekind's $\eta$-function and Rogers–Ramanujan identities}, Bulletin of the London Mathematical Society \textbf{44}  (2012), 1--11.
	
    \bibitem{Za} D. Zagier, {\em Vassiliev invariants and a strange identity related to the Dedekind eta-function}, Topology \textbf{40} (2001), 945–960.
        
    \bibitem{Za2006} D. Zagier, {\em The Mellin transform and other useful analytic techniques, Appendix to E. Zeidler, Quantum Field Theory I: Basics in Mathematics and Physics. A bridge between mathematicans and physicists,} Springer-Verlag, Berlin-Heidelberg-New York (2006), 305--323.
    
    \bibitem{ZaQ} D. Zagier, {\em Quantum modular forms}, Quanta of Mathematics {\bf 11} (2010), 659--675.
    
    \bibitem{Za20} D. Zagier, {\it Holomorphic quantum modular forms lecture}, special program ``Dynamics: Topology and numbers at the Hausdorf Center for Mathematics'', Spring 2020.
    
   	\bibitem{ZW2} S. Zwegers, {\em Mock theta functions}, PhD thesis, Utrecht University, 2002.
		
\end{thebibliography}
\end{document}